\documentclass{amsart}

\usepackage{amssymb,amsmath,cite}

\usepackage{amsfonts}
\usepackage{enumerate}
\usepackage{mathrsfs}
\usepackage{amscd}
\usepackage{amsmath}
\usepackage{latexsym}
\usepackage{amssymb}
\usepackage{amsthm}
\usepackage{graphicx}
\usepackage{color}
\usepackage{cite}
\usepackage{bm}

\newcommand{\al}{\alpha}
\newcommand{\be}{\beta}

\newcommand{\la}{\lambda}

\newcommand{\si}{\sigma}
\newcommand{\Si}{\Sigma}

\newcommand{\ga}{\gamma}
\newcommand{\om}{\omega}
\newcommand{\Om}{\Omega}
\newcommand{\ka}{\kappa}

\makeatletter
\@addtoreset{equation}{section}

\newtheorem{proposition}{Proposition}[section]
\newtheorem{definition}{Definition}[section]
\newtheorem{lemma}{Lemma}[section]
\newtheorem{theorem}{Theorem}[section]
\newtheorem{corollary}{Corollary}[section]

\theoremstyle{remark}
\newtheorem{remark}[theorem]{Remark}

\allowdisplaybreaks

\begin{document}

\title[]{Almost sure global weak solutions and optimal decay for the incompressible generalized Navier-Stokes equations}

\author{Y.-X. Lin}
\address{Yuan-Xin Lin
\newline\indent
School of Mathematical Sciences, Shanghai Jiao Tong University,
Shanghai, P. R. China}
\email{yuanxinlin@sjtu.edu.cn}

	\author{Y.-G. Wang}
\address{Ya-Guang Wang
	\newline\indent
	School of Mathematical Sciences,  MOE-LSC and SHL-MAC, Shanghai Jiao Tong University,
	Shanghai, P. R. China}
\email{ygwang@sjtu.edu.cn  }
\maketitle

\date{}
\date{}

\begin{abstract} 
In this paper, we consider the initial value problem of the incompressible generalized Navier-Stokes equations with initial data being in negative order Sobolev spaces, in the whole space $\mathbb{R}^d$ with $d \geq 2$. The generalized Navier-Stokes equations studied here is obtained by replacing the standard Laplacian in the classical Navier-Stokes equations by the fractional order Laplacian $-(-\Delta)^\al$ with $\al \in \left( \frac{1}{2},\frac{d+2}{4} \right]$. After an appropriate randomization on the initial data, we obtain the almost sure existence and optimal decay rate of global weak solutions when the initial data belongs to $\Dot{H}^s(\mathbb{R}^d)$ with $s\in (-\al+(1-\al)_+,0)$. Moreover, we show that the weak solutions are unique when $\al=\frac{d+2}{4}$ with $d \geq 2$.

\end{abstract}
~\\
{\textbf{\scriptsize{2020 Mathematics Subject Classification:}}\scriptsize{ 35Q30, 35R60, 76D03}}.
~\\
\textbf{\scriptsize{Keywords:}} {\scriptsize{Generalized Navier-Stokes equations, almost sure existence of global weak solutions, optimal decay rate, negative order Sobolev datum.}}

\section{Introduction}
The aim of this work is to study the following initial value problem for the generalized incompressible Navier-Stokes equations with datum being in negative order Sobolev spaces in $\{t>0, x\in  \mathbb{R}^d \}$ with $d \geq 2$:
\begin{equation}\label{GNS}
    \left\{
    \begin{aligned}
       \partial_t u + (-\Delta)^\al u + (u \cdot \nabla) u + \nabla p & = 0, \quad \quad x \in \mathbb{R}^d, ~\  t>0,  \\
        \mathrm{div} \, u & =0, \\
        u(0,x) & = u_0(x),
    \end{aligned}
\right.
\end{equation}
where $u=(u_1, \cdots , u_d)^T$ and $p$ denote the velocity field and the pressure respectively. The fractional order Laplacian $(-\Delta)^\al$ is defined via the Fourier transform
\begin{equation*}
    (-\Delta)^\al u (t,x) = \mathcal{F}^{-1}\Big\{ |\xi|^{2\al} [\mathcal{F}u(t,\cdot)](\xi) \Big\}(x),
\end{equation*}
where $\mathcal{F}$ and $\mathcal{F}^{-1}$ denote the partial Fourier transform and inverse Fourier transform in the space variables, respectively.

The study of weak solutions to the classical Navier-Stokes equations, i.e. $\eqref{GNS}_1$ with $\al =1$, can date back to 1930s. For any given divergence-free $L^2(\mathbb{R}^d)$ initial data $u_0$, Leray \cite{Leray} obtained the existence of global weak solutions to the Cauchy problem, and the bounded domain counterpart was established by Hopf \cite{Hopf}. Later on, the existence of uniformly locally square integrable weak solutions was given in Lemarié-Rieusset \cite{LR}. These well-posedness results are obtained by assuming that the initial data belongs to the sub-critical or the critical spaces. Burq and Tzvetkov \cite{BT} deduced that the cubic nonlinear wave equation is strongly ill-posed when the initial data lies in the super-critical space, however, after an appropriate randomization on the initial data, they showed that the problem is locally well-posed for almost sure randomized initial data. Since then, their approach was widely used in studying other evolutional nonlinear PDEs with super-critical initial data, see \cite{BT1,DH,QTW,WH} and the references therein. In particular, there has been interesting progress on the Cauchy problem for the classical Navier-Stokes equations when the initial data belongs to $\Dot{H}^s(\mathbb{T}^d)$ or $\Dot{H}^s(\mathbb{R}^d)$ with $s \leq 0$. When $u_0 \in L^2(\mathbb{T}^3)$, Zhang and Fang in \cite{ZF} obtained the almost sure existence and uniqueness of a local strong solution, and it exists globally in time in a positive probability providing that the $L^2$ norm of $u_0$ is small. Meanwhile, as $u_0 \in L^2(\mathbb{T}^3)$, Deng and Cui in \cite{DC} showed the existence and uniqueness of a local mild solution, and it exists globally in time in a large probability with small initial datum. Nahmod, Pavlovic and Staffilani \cite{NPS} obtained the almost sure existence of global weak solutions for  randomized initial data when $u_0 \in \Dot{H}^s(\mathbb{T}^d)$ with $d=2,3$ and $-\frac{1}{2^{d-1}}<s<0$, and they also showed that the weak solutions are unique in dimension two. In the whole space case, the same results as in \cite{NPS} were obtained by Chen, Wang, Yao and  Yu \cite{CWYY}. For $u_0 \in \Dot{H}^s(\mathbb{R}^3)$ with $-\frac{1}{2} \leq s <0$, Wang and Wang \cite{WW} got the existence of global weak solutions for almost sure randomized initial data. Afterwards, Du and Zhang \cite{DZ} obtained the almost sure existence of global weak solutions when the initial data $u_0$ belongs to $\Dot{H}^s(\mathbb{T}^d)$ with $-1<s<0$ for all $d \geq 2$, and the uniqueness holds as well when  $d=2$.

There are also many interesting results on the time decay of weak solutions for the classical Navier-Stokes equations, which give a confirmed answer to the question posed by Leray \cite{Leray}, i.e. whether  the $L^2$ norm of weak solutions tends to zero as time goes to infinity. The Fourier splitting method was introduced by Schonbek \cite{S1} for studying the decay rates for parabolic conservation laws. Later on, she used this method to obtain the first algebraic decay rate for the 3D Navier-Stokes equations \cite{S2}, more precisely, she showed that for any $u_0 \in L^1 \cap L^2$, the Leray-Hopf weak solutions satisfy
\begin{equation}\label{D}
    ||u(t)||_{L_x^2} \leq C(t+1)^{-\be},
\end{equation}
with $\be=\frac{1}{4}$ and $C$ being a positive constant. Afterwards, Schonbek \cite{S3}, Kajikiya and Miyakawa \cite{KM} improved the decay exponent $\be$ to be $\frac{3}{2}\left( \frac{1}{p}-\frac{1}{2} \right)$ when the initial data $u_0$ belongs to $L^p(\mathbb{R}^3) \cap L^2(\mathbb{R}^3)$ with $p \in [1,2)$. Wiegner \cite{W} obtained the optimal decay rate for weak solutions and the difference between weak solutions and the corresponding heat flow in a general setting. More recently, for $u_0 \in \Dot{H}^s(\mathbb{R}^3)$ with $s \in \left[ -\frac{1}{2},0 \right)$, Wang and Yue got the almost sure decay rate \eqref{D} with $\be=-\frac{s}{2}$. We refer the readers to \cite{BM,KO,M,K,MS} for more results on large time behavior for the classical Navier-Stokes equations, and \cite{WS,DLM1,DLM2,KMV,Y} for other fluid models.

There is also certain progress on problems of generalized Navier-Stokes equations. For any $\al >0$ and $u_0 \in L^2(\mathbb{T}^d)$ with $d \geq 2$, Wu \cite{Wu} obtained the existence of global weak solutions to the problem \eqref{GNS}. It is well known that the critical exponent for the equation $\eqref{GNS}_1$ is $\al = \frac{d+2}{4}$. In the sub-critical and critical case, i.e. $\al \geq \frac{d+2}{4}$, Wu in \cite{Wu} also showed the existence of a unique strong solution to the problem \eqref{GNS} for initial data $u_0$ being in $H^s(\mathbb{T}^d)$ with $d \geq 2$ and $s \geq 2\al$. In recent years, the problems of the super-critical case $0<\al < \frac{d+2}{4}$ attract more attention. When $d \geq 3$, $\al \in \left( \frac{1}{2},\frac{d+2}{4} \right)$ and $u_0 \in \Dot{H}^s(\mathbb{T}^d)$ or $\Dot{H}^s(\mathbb{R}^d)$ with
\begin{equation*}
    s \in \left\{ 
    \begin{aligned}
        (1-2\al, 0] & ,~\ \ \ \ \ \ \frac{1}{2} < \al \leq 1, \\
        [-\al, 0] & ,~\ \ \ \ \ \ 1< \al < \frac{d+2}{4},
    \end{aligned}
    \right.
\end{equation*}
Zhang and Fang \cite{ZF1} obtained the local existence of mild solutions to the problem \eqref{GNS} for almost sure randomized initial data, and the solutions exist globally in time in a positive probability providing that the $\Dot{H}^s$ norm of $u_0$ is sufficiently small. Recently, the authors \cite{LW} showed the almost sure existence of global weak solutions to the problem \eqref{GNS} when $\al \in \left( \frac{2}{3},1 \right]$, $u_0 \in \Dot{H}^s(\mathbb{T}^d)$ with $d \geq 2$ and $s \in (1-2\al,0)$. In addition, Jiu and Yu \cite{JY} obtained the optimal decay rate for the $L^2$ norm of the global weak solutions when $u_0 \in L^p(\mathbb{R}^3) \cap L^2(\mathbb{R}^3)$ with $1 \leq p < 2$ and $\al \in \left( 0,\frac{5}{4} \right)$. Later on, the decay result was extended to $d \geq 2$, $\al \in \left( 0,\frac{d+2}{4} \right]$ by Yu \cite{Y}, and $d \geq 2$, $\al \in (0,1)$ by Liu \cite{L}.

It is worth emphasizing that the lower bound $\frac{2}{3}$ of $\al$ in \cite{LW} comes from technical constraints. Indeed, when $t$ near zero, we obtain a mild solution in a subspace of $L^{\frac{4}{1+2\ga}} \left( 0,\tau; L_x^{\frac{2d}{(3-2\ga)\al - 2}} \right)$ with $\ga \in \left[ 0,\frac{1}{2} \right)$ for small $\tau>0$. In order to ensure that $\frac{2d}{(3-2\ga)\al - 2}$ makes sense one needs $\al > \frac{2}{3}$. However, we observed that the restriction $\al > \frac{2}{3}$ can be removed by using a fixed point argument in a proper subspace of $L^a \left( 0,\tau; L^p(\mathbb{R}^d) \right)$, where $a$ and $p$ are defined in \eqref{AP}. Thus, we shall consider $\al > \frac{1}{2}$, which is the natural lower bound since we need to transform the problem to its integral form to overcome the singularity at $t=0$.

In this paper, we shall first extend the almost sure global existence result obtained in \cite{LW} to the case $\al \in \left( \frac{1}{2},\frac{d+2}{4} \right]$ when $d \geq 2$, moreover, the space domain shall be the whole space $\mathbb{R}^d$ instead of the torus $\mathbb{T}^d$. Next, we shall study the optimal time decay rate for the $L^2$ norm of global weak solutions and of the difference between the solution of \eqref{GNS} 
and the corresponding generalized heat flow. Finally, when $\al=\frac{d+2}{4}$ with $d \geq 2$, we shall obtain the uniqueness of weak solutions, which extends the uniqueness result of the 2D classical Navier-Stokes equations deduced in \cite{NPS,DZ}.

We follow the way given in \cite{WW,ZF,DZ} to introduce the randomization set up in $\mathbb{R}^d$. Let $\phi \in C_c^\infty (\mathbb{R}^d)$ be a non-negative and radial symmetric function, satisfying
\begin{equation*}
    \phi (\xi) := \left\{
    \begin{aligned}
        & 1, ~\ |\xi|<\frac{3}{4}, \\
        & 0, ~\ |\xi| \geq 1.
    \end{aligned}
    \right.
\end{equation*}
Define
$$\varphi (\xi) = \frac{\phi(\xi)}{\sum_{k \in \mathbb{Z}^d} \phi(\xi-k)},$$
It is clear that $\varphi$ is also  non-negative and radial symmetric, and supported in $\overline{B(0,1)} := \{ \xi \in \mathbb{R}^d \, : \, |\xi| \leq 1 \}$.

Define the differential operator $\varphi (\Lambda - k)$ by
$$\varphi(\Lambda - k)f(x) = \mathcal{F}^{-1} \Big\{ \varphi (\xi - k) \mathcal{F}f (\xi) \Big\}(x)$$
with $\mathcal{F}$ and $\mathcal{F}^{-1}$ being the Fourier and inverse Fourier respectively of the related function in the space variable.
Since $\sum_{k \in \mathbb{Z}^d} \varphi (\xi - k) = 1$ for all $\xi\in \mathbb{R}^d$, one can decompose $f$ as
\begin{equation*}
    f(x) = \sum_{k \in \mathbb{Z}^d} \varphi (\Lambda - k) f(x).
\end{equation*}

Let $\{g_k(\omega)\}_{k \in \mathbb{Z}^d}$ be a sequence of independent, mean zero, real-valued random variables defined on a given probability space $(\Omega,\mathcal{F}, \mathbb{P})$, such that
\begin{equation}\label{HN}
    \exists C>0, \,\,\,\, \forall k \in \mathbb{Z}^d,\,\,\,\, \int_\Omega|g_k(\omega)|^{2n}\mathbb{P}(d\omega) \leq C,
\end{equation}
where $n$ is an integer satisfying
\begin{equation}\label{rs}
    2n \geq r_s := \max\left\{ \frac{4\al}{2\al(\mu+1)-1}, \frac{2d}{2\al(1-\mu)-1} \right\},
\end{equation}
with
\begin{equation}\label{mu}
        \mu = \left( -\frac{s}{\al}+\frac{1}{2\al}-1 \right)_+ := \max\left\{ -\frac{s}{\al}+\frac{1}{2\al}-1,0 \right\}.
\end{equation}

In the following discussion, we shall always denote by $X_\si$ the divergence-free subspace of $X$ for any Banach space $X$ of functions defined in $\mathbb{R}^d$. 

For $u_0 \in \Dot{H}_\si^s(\mathbb{R}^d)$, we introduce its randomization by
\begin{equation}\label{RI}
    u_0^\omega =  \sum_{k \in \mathbb{Z}^d} g_k(\omega) \varphi (\Lambda - k) u_0,
\end{equation}
which defines a measurable map from $(\Omega,\mathcal{F}, \mathbb{P})$ to $\Dot{H}_\si^s(\mathbb{R}^d)$ equipped with the Borel $\sigma$-algebra, and one can verify that $u_0^\omega \in L^2(\Omega, \Dot{H}_\si^s(\mathbb{R}^d))$.

Recall that for any Banach space $X$, $f \in C_{weak}([0,T];X)$ means that for any $\psi \in X'
    $, the dual space of $X$, $\langle f(t), \psi \rangle_{X \times X'}$ is continuous on $[0,T]$. The weak solution of the problem \eqref{GNS} is defined as the following one.

\begin{definition}\label{DW}
    Let $\al \in \left( \frac{1}{2},\frac{d+2}{4} \right]$ and $u_0 \in \Dot{H}_\si^{s}(\mathbb{R}^d)$ with $d \geq 2$ and $s \in (-\al+(1-\al)_+,0)$. A function $u \in L_{loc}^\infty ((0,T];L_\si^2(\mathbb{R}^d)) \cap L_{loc}^2((0,T];\Dot{H}_\si^\al(\mathbb{R}^d))$ is a weak solution of the problem \eqref{GNS} on $[0,T]$, if for any divergence-free test function $\psi \in C_c^\infty(\mathbb{R}^d)$ and a.e. $t \in [0,T]$, it holds that
    \begin{equation*}
        \langle \partial_t u,\psi \rangle(t) + \langle \Lambda^\al u,\Lambda^\al \psi \rangle(t) - \langle (u \cdot \nabla \psi), u \rangle (t)= 0,
    \end{equation*}
    and
    \begin{equation*}
        \lim_{t \rightarrow 0^+} ||u(t)-u_0||_{\Dot{H}_x^s} = 0,
    \end{equation*}
    where $(\Lambda^\al u) (x) = \mathcal{F}^{-1}\Big\{ |\xi|^{\al} \mathcal{F}u(\xi) \Big\}(x)$, the notation $\langle \cdot, \cdot \rangle$ represents the dual pairing with respect to the space variables.
    
    We say that the problem \eqref{GNS} has a global weak solution if for any $T>0$, it admits a weak solution on $[0,T]$.
\end{definition}

The main results of this paper are the following two theorems, one is the almost sure existence of a global weak solution of the problem \eqref{GNS} after randomization of the initial data, and the other one concerns the optimal decay rate of the weak solution as $t\to +\infty$.

\begin{theorem}\label{TW}
    Let $\al \in \left( \frac{1}{2},\frac{d+2}{4} \right]$ with $d \geq 2$ and $s \in (-\al+(1-\al)_+,0)$. For any given $u_0 \in \Dot{H}_\si^s(\mathbb{R}^d)$, let $u_0^\om$ be the randomization  of $u_0$ given in \eqref{RI}. Then there exists an event $\Sigma \subset \Om$ with 
probability $1$, i.e. $\mathbb{P}(\Sigma) = 1$, such that for any $\om \in \Sigma$, the problem \eqref{GNS} with the initial data $u_0^\om$ admits a global weak solution in the sense of Definition \ref{DW} of the form $u = h^\om + w,$ with $h^\om = e^{-t(-\Delta)^\al}u_0^\om$, and satisfying for any $T>0$,
    \begin{equation}\label{RH1}
        (1 \land t )^\mu w \in C_{weak} \Big([0,T];L_\si^2(\mathbb{R}^d) \Big) \cap L^2 \Big(0,T;\Dot{H}_\si^\al(\mathbb{R}^d) \Big),
    \end{equation}
    and
    \begin{equation}\label{RH2}
        w \in L^{\frac{4\al}{2\al(\mu+1)-1}}\left( 0,T;L^{\frac{2d}{d+1-2\al(\mu+1)}}(\mathbb{R}^d) \right),
    \end{equation}
    where $\mu$ is defined in \eqref{mu} and $1 \land t = \min\{ 1,t \}$.

    Moreover, when $\al = \frac{d+2}{4}$ with $d \geq 2$, the weak solution satisfying \eqref{RH1}-\eqref{RH2} is unique.
\end{theorem}

\begin{theorem}\label{TD}
    Let $u = h^\om + w$ be the almost sure global weak solution of the problem \eqref{GNS} with the initial data $u_0^\om$ obtained in Theorem \ref{TW} for all $\om \in \Sigma$ with $\mathbb{P}(\Sigma) = 1$.  Then there exist $T_0^\om > e$ for each $\omega\in \Sigma$, and a positive constant $C$ depending only on $\al,d,s$ and $\om$, such that
    \begin{equation}\label{TD1}
        ||u(t)||_{L_x^2}^2 \leq Ct^{\frac{s}{\al}}, \quad \forall t>T_0^\om,
    \end{equation}
    and
    \begin{equation}
        ||w(t)||_{L_x^2}^2 = ||u(t)-h^\om(t)||_{L_x^2}^2 \leq Ct^{-\frac{d+2}{2\al}+2+\frac{2s}{\al}}, \quad \forall t>T_0^\om.
    \end{equation}
\end{theorem}

\begin{remark}
    (1) When $\al \in \left( 1,\frac{d+2}{4} \right]$ with $d \geq 3$, the property \eqref{RH1} in Theorem \ref{TW} can be improved as
    $$w \in C_{weak} \Big([0,T];L_\si^2(\mathbb{R}^d) \Big) \quad \mathrm{and} \quad (1 \land t )^\mu w \in L^2 \Big(0,T;\Dot{H}_\si^\al(\mathbb{R}^d) \Big).$$

    (2) Theorems \ref{TW} and \ref{TD} also hold true if  the equations $\eqref{GNS}_1$ is defined in $\mathbb{T}^d$ instead of $\mathbb{R}^d$, under the randomization introduced in \cite{DZ,ZF,LW}. As a consequence, we extend our previous result \cite[Theorem 1.1]{LW} from $\al \in \left( \frac{2}{3},1 \right]$ to $\al \in \left(\frac{1}{2},\frac{d+2}{4} \right]$ for all $d \geq 2$.

    (3) The decay rate obtained in Theorem \ref{TD} is optimal since the corresponding generalized heat flow $h^\om(t)$ has the same decay rate as in \eqref{TD1}.
\end{remark}

The remainder of this paper is organized as follows. In Section 2, we give serval notations and properties of the solution spaces, and estimates of the generalized heat flow $h^\om$, which shall be used later. The existence of global weak solutions to the problem \eqref{GNS} is obtained in Section 3. In Section 4, we show the uniqueness of weak solutions when $\al=\frac{d+2}{4}$, and in Section 5 we obtain the optimal time decay given in Theorem \ref{TD}. In Appendix \ref{A}, we present the proofs of serval lemmas used in this paper, for completeness.

Throughout this paper, the notation $A \lesssim B$ denotes $A \leq CB$ for some positive constant $C$, which may change from line to line.

\section{Preliminary}

In this section, we first introduce the inhomogeneous and homogeneous Sobolev spaces, and give serval properties used frequently in the following. Then, the required estimates for the generalized heat flow $h^\om=e^{-t(-\Delta)^\al}u_0^\om$ are established.

\subsection{Notations and several lemmas}
$~\ $

For any $\be \in \mathbb{R}$ and $r \in (1,\infty)$, define the inhomogeneous Sobolev spaces by
\begin{equation}\label{S}
    W^{\be,r}(\mathbb{R}^d) := \{ f \in \mathcal{S}' : (1-\Delta)^{\frac{\be}{2}} f \in L^r(\mathbb{R}^d) \},
\end{equation}
endowed with the norm $||f||_{W_x^{\be,r}} := ||(1-\Delta)^{\frac{\be}{2}} f||_{L_x^r}$, where $\mathcal{S}'$ is the Schwartz distribution space. The homogeneous counterpart of $W^{\be,r}(\mathbb{R}^d)$ is defined by
\begin{equation}\label{HS}
    \Dot{W}^{\be,r}(\mathbb{R}^d) := \{ f \in \mathcal{S}' : \Lambda^{\be} f \in L^r(\mathbb{R}^d) \},
\end{equation}
with the semi-norm $||f||_{\Dot{W}_x^{\be,r}} := ||\Lambda^\be f||_{L_x^r}$, and $\Lambda := (-\Delta)^{\frac{1}{2}}$. Clearly, for $\be>0$ one has
\begin{equation}\label{N}
    ||f||_{W_x^{\be,r}} \sim ||f||_{L_x^r} + ||f||_{\Dot{W}_x^{\be,r}},
\end{equation}
and the dual space of $W^{\be,r}(\mathbb{R}^d)$ is $W^{-\be,r'}(\mathbb{R}^d)$, with $\frac{1}{r}+\frac{1}{r'}=1$.

The following estimate for fractional order derivative of product functions will be used later.

\begin{lemma}\label{PR}\cite[Lemma 3.1]{J}
    For given $p, r_1, r_2 \in (1,\infty)$ satisfing
    \begin{equation*}
        \frac{1}{p} = \frac{1}{r_1} + \frac{1}{q_1} = \frac{1}{r_2} + \frac{1}{q_2},
    \end{equation*}
    then for any fixed  $\nu>0$, there exists a positive constant $C$ such that
    \begin{equation*}
        ||\Lambda^\nu (fg)||_{L^p(\mathbb{R}^d)} \leq C \left( ||f||_{L^{q_1}(\mathbb{R}^d)} ||\Lambda^\nu g ||_{L^{r_1}(\mathbb{R}^d)} +||g||_{L^{q_2}(\mathbb{R}^d)} ||\Lambda^\nu f||_{L^{r_2}(\mathbb{R}^d)} \right).
    \end{equation*}
\end{lemma}

Now, we recall the following space-time estimate for the generalized heat flow.

\begin{lemma}\label{PQ}\cite[Lemma 3.1]{MYZ}
    Let $1\leq q \leq p \leq \infty$ and $f \in L^q(\mathbb{T}^d)$, then for any $\nu \geq 0$, there exists a positive constant $C$ such that
    \begin{equation*}
        ||\Lambda^\nu e^{-t(-\Delta)^\al} f||_{L^p(\mathbb{R}^d)} \leq C t^{-\frac{\nu}{2\al} - \frac{d}{2\al}\left(\frac{1}{q}-\frac{1}{p}\right)}||f||_{L^q(\mathbb{R}^d)}.
    \end{equation*}
\end{lemma}

Using Lemma \ref{PQ} and the equivalence relation \eqref{N}, we immediately obtain that for any $\nu >0$, 
\begin{equation}\label{IH}
    \begin{aligned}
        ||e^{-t(-\Delta)^\al} f||_{W^{\nu,p}(\mathbb{R}^d)} & = ||(1-\Delta)^{\frac{\nu}{2}} e^{-t(-\Delta)^\al} f||_{L^p(\mathbb{R}^d)} \\
        & \lesssim  ||e^{-t(-\Delta)^\al} f||_{L^p(\mathbb{R}^d)} + ||\Lambda^\nu e^{-t(-\Delta)^\al} f||_{L^p(\mathbb{R}^d)} \\
        & \lesssim t^{- \frac{d}{2\al}\left(\frac{1}{q}-\frac{1}{p}\right)}||f||_{L^q(\mathbb{R}^d)} + t^{-\frac{\nu}{2\al} - \frac{d}{2\al}\left(\frac{1}{q}-\frac{1}{p}\right)}||f||_{L^q(\mathbb{R}^d)} \\
        & \lesssim \left( 1+t^{\frac{\nu}{2\al}}\right) t^{-\frac{\nu}{2\al} - \frac{d}{2\al}\left(\frac{1}{q}-\frac{1}{p}\right)}||f||_{L^q(\mathbb{R}^d)}.
    \end{aligned}
\end{equation}

\begin{lemma}\label{EY}\cite[Theorem 4 in Appendix]{S}
    Assume that $0<\tau<1$ and $1<p<q<\infty$ satisfy
    \begin{equation*}
        1-(\frac{1}{p}-\frac{1}{q})=\tau,
    \end{equation*}
    and define
    \begin{equation*}
        (I_\tau f)(t)=\int_{-\infty}^\infty f(s)|t-s|^{-\tau} ds,
    \end{equation*}
    then there exists a positive constant C depending only on $p$ and $q$ such that
    \begin{equation*}
        ||I_\tau f||_{L^q(\mathbb{R})} \leq C ||f||_{L^p(\mathbb{R})}.
    \end{equation*}
\end{lemma}

As we knew, the standard heat kernel $e^{t\Delta}$ has the maximal regularity estimate \cite[Theorem 7.3]{LR}. The following one   is also true for the generalized heat equation, one can see \cite[pp.1654-1655,1657]{HP} for the detail.

\begin{lemma}\label{MR}\cite[Lemma 2.2]{D}
    For any fixed $\al>0$, the operator 
    \begin{equation*}
        F \mapsto \int_0^t e^{-(t-s)(-\Delta)^{\al}} \Lambda^{2\al} F(s,x) ds,
    \end{equation*}
    is bounded from $L_T^q L^r(\mathbb{R}^d)$ to
    $L_T^q L^r(\mathbb{R}^d)$ for any given  $T>0$ and $q,r \in (1,\infty)$, where $L_T^q L^r(\mathbb{R}^d)$ denotes $L^q(0,T; L^r(\mathbb{R}^d))$ for simplicity.
\end{lemma}

\begin{corollary}\label{MR1}\cite[Corollary 2.1]{LW}
    For any given $\zeta \geq 2\al > 0$, $T>0$ and $q,r \in (1,\infty)$,  the operator 
    \begin{equation*}
        F \mapsto \int_0^{\frac{t}{2}} t^{\frac{\zeta}{2\al}-1} e^{-(t-s)(-\Delta)^{\al}} \Lambda^\zeta F(s,x) ds,
    \end{equation*}
    is bounded from $L_{\frac{T}{2}}^q L^r(\mathbb{R}^d)$ to
    $L_T^q L^r(\mathbb{R}^d)$.
\end{corollary}

The following one extends the well-known result \cite[Lemma 14.1]{LR} of the classical heat kernel.

\begin{lemma}\label{ML}\cite[Remark 2.1]{LW}
    For given $\al>0$ and $T>0$, the operator
    \begin{equation*}
        F \mapsto \int_0^t e^{-(t-s)(-\Delta)^\al} \Lambda^\al F(s,x) ds,
    \end{equation*}
    is bounded from $L_T^2 L^2(\mathbb{R}^d)$ to $L_T^\infty L^2(\mathbb{R}^d)$.
\end{lemma}

Finally, we recall the Serrin-type uniqueness result for the generalized Navier-Stokes equations, which was given in \cite[Theorem 3.2]{LW} when the space domain is the torus $\mathbb{T}^d$ with $d \geq 2$. From that proof, one can see that it is also true for the whole space $\mathbb{R}^d$  case. 

\begin{lemma}
    \label{UR}
    Let $\al > 0$ and $0 \leq \tau_1 < \tau_2$ be fixed. For any
    given divergence-free field $v_0 \in L^2(\mathbb{R}^d)$,  and 
    $$\psi \in L^\infty \left(\tau_1,\tau_2; L^2(\mathbb{R}^d)\right) \cap L^2 \left(\tau_1,\tau_2; H^\al(\mathbb{R}^d)\right) \cap L^m \left(\tau_1,\tau_2; W^{\be,r}(\mathbb{R}^d)\right),$$ 
    with $\frac{2\al}{m} + \frac{d}{r} = 2\al-1+\be$ for some
    \begin{equation}\label{be}
        \left\{
        \begin{aligned}
            & \be \in \left[1-\al, \frac{d}{r}+1-\al \right), \,\, r \in \left( \frac{d}{\al}, \infty \right), \qquad 0< \al \leq 1, \\
            & \be=0, \,\, r \in \left( \frac{d}{\al},\frac{d}{\al-1} \right], \qquad \qquad \qquad \qquad \qquad \,\, \al >1,
        \end{aligned}
        \right.
    \end{equation}
    assume that $v_1, v_2 \in L^\infty \left(\tau_1,\tau_2; L^2(\mathbb{R}^d)\right) \cap L^2 \left(\tau_1,\tau_2; H^\al(\mathbb{R}^d) \right)$ are two weak solutions of the following problem in the sense of $\mathcal{D}'(\mathbb{R}^d)$,
    \begin{equation}\label{UR11}
        \left\{
        \begin{aligned}
            & \partial_t v + (-\Delta)^\al v + B(v,v) + B(v,\psi) + B(\psi,v) + B(\psi,\psi) = 0, \\
            & \mathrm{div}\,v=0,\\
            &  \lim_{t \rightarrow {\tau_1}+0} ||v(t) - v_0||_{L^2(\mathbb{R}^d)} = 0.
          \end{aligned}
        \right.
    \end{equation}
    where $B(f,g)=P(f\cdot\nabla)g$  with $P$ being the Leray projection.  
    Furthermore, if $v_1 \in L^m \left(\tau_1,\tau_2; W^{\be,r}(\mathbb{R}^d) \right)$ and $v_2$ satisfies
    \begin{equation}\label{UR1}
    ||v_2(t)||_{L_x^2}^2 + 2\int_{\tau_1}^t ||\Lambda^\al v_2|| ds \leq ||v_2(\tau_1)||_{L_x^2}^2 + 2\int_{\tau_1}^t \langle B(v_2,v_2),\psi \rangle + \langle B(\psi,v_2),\psi \rangle ds,
    \end{equation}
    for any $t \in (\tau_1,\tau_2]$, then $v_1=v_2$ on $[\tau_1,\tau_2]$.
\end{lemma}

Specially, we have

\begin{corollary}\label{UR-1}
    Under the notations and hypotheses given in Lemma \ref{UR}, assume that
    $$v_1,v_2 \in L^\infty \left(\tau_1,\tau_2; L^2(\mathbb{R}^d)\right) \cap L^2 \left(\tau_1,\tau_2; H^\al(\mathbb{R}^d) \right) \cap L^m \left(\tau_1,\tau_2; W^{\be,r}(\mathbb{R}^d) \right),$$
    are two weak solutions to the problem \eqref{UR11}, then $v_1=v_2$ on $[\tau_1,\tau_2]$.
\end{corollary}

\subsection{Estimates of the generalized heat flow $h^\om$.}
$~\ $

In this subsection, we give several estimates of the generalized heat flow $h^\omega=e^{-t(-\Delta)^\al} u_0^\omega$ defined in Theorem \ref{TW}. Before that, first we have

\begin{lemma}\label{NW}
    Let $\{g_k(\omega)\}_{k\in \mathbb{Z}^d}$ be a sequence of independent, mean zero, complex valued random variables defined on a given probability space $(\Omega,\mathcal{F}, \mathbb{P})$ satisfying \eqref{HN}, then there exists some positive constant $C$ such that for any $\{c_k\}_{k \in \mathbb{Z}^d} \in l^2$, one has
    \begin{equation*}
        \left|\left| \sum_{k \in \mathbb{Z}^d} c_k g_k(\omega) \right|\right|_{L_\omega^{r_s}} \leq C \left( \sum_{k \in \mathbb{Z}^d} |c_k|^2 \right)^\frac{1}{2},
    \end{equation*}
    with $r_s$ being given in \eqref{rs}.
\end{lemma}

    This result can be directly obtained by using Lemma 4.2 given in \cite{BT}.

Now, let us consider $h^\omega=e^{-t(-\Delta)^\al} u_0^\omega$, the solution to the following linear problem
\begin{equation}\label{FE}
    \left\{
    \begin{aligned}
        \partial_t h + (-\Delta)^\al h & = 0, \\
        h(0) & = u_0^\om,
    \end{aligned}
    \right.
\end{equation}
we have the following results.

\begin{lemma}\label{SE}
    Assume that $\{g_k(\omega)\}_{k\in \mathbb{Z}^d}$ satisfies the condition given in \eqref{HN}, $u_0 \in \Dot{H}_\si^s(\mathbb{R}^d)$ with $d \geq 2$ and $s \in (-\al+(1-\al)_+,0)$, and one of the following holds
    
    (1) $2\leq a, p \leq r_s < \infty$ and $\rho,\eta \in \mathbb{R}$ satisfying $\eta - 2 \al \rho -\frac{2\al}{a} \leq s$, $\rho a > -1$.

    (2) $a=\infty$, $2 \leq p \leq r_s$ and $\rho , \eta \in \mathbb{R}$ satisfying $\eta - 2 \al \rho \leq s$.
    
\noindent Then there exists a positive constant $C$ such that
    \begin{equation}\label{SE1}
        ||h^\omega||_{L_\omega^{r_s} L_{\rho;T}^a \Dot{W}_x^{\eta,p} } \leq C T^\si ||u_0||_{\Dot{H}^s},
    \end{equation}
    where $\si = {\frac{1}{2\al}\left( s-\left( \eta - 2\al \rho - \frac{2\al}{a} \right) \right)} \geq 0$, and
    $\|f\|_{L_{\rho;T}^a}$ denotes
    $\|t^\rho f(t)\|_{L^a(0, T)}$.
    Furthermore, we have
    \begin{equation}\label{SE2}
        \mathbb{P}(E_{\la,T}) \leq C T^{\si r_s} \frac{||u_0||_{\Dot{H}^s}^{r_s}}{\la^{r_s}}, \quad \forall \lambda>0,
    \end{equation}
    where $E_{\la,T}=\left\{ \omega \in \Omega \, : \, ||h^\omega||_{L_{\rho;T}^a \Dot{W}_x^{\eta,p} } \geq \la \right\}$.
\end{lemma}
\begin{proof}

(1) When $2 \leq a \leq r_s < \infty$, by using the generalized Minkowski's inequality and Lemma \ref{NW}, one gets

\begin{equation}\label{H1}
    \begin{aligned}
        ||h^\omega||_{L_\omega^{r_s} L_{\rho; T}^a \Dot{W}_x^{\eta,p}} & \lesssim \left|\left|   \sum_{k \in \mathbb{Z}^d} g_k(\om) t^\rho \Lambda^\eta e^{-t(-\Delta)^\al} \varphi (\Lambda - k) u_0 \right|\right|_{L_T^a L_x^{p} L_\omega^{r_s}} \\
    & \lesssim \left( \sum_{k \in \mathbb{Z}^d} \left|\left| t^\rho \Lambda^\eta e^{-t(-\Delta)^\al} \varphi (\Lambda - k) u_0 \right|\right|_{L_T^a L_x^{p}}^2 \right)^{\frac{1}{2}}.
    \end{aligned}
\end{equation}
Using Bernstein's inequality and Plancherel's identity, we obtain
\begin{align*}
    \left|\left|e^{-t(-\Delta)^\al} \Lambda^\eta \varphi (\Lambda - k) u_0 \right|\right|_{L_x^p} &
    \leq C \left|\left|e^{-t(-\Delta)^\al} \Lambda^\eta \varphi (\Lambda - k) u_0 \right|\right|_{L_x^2} \\
        & = C \left|\left|e^{-t|\xi|^{2\al}} |\xi|^\eta \varphi (\xi - k) \widehat{u_0}(\xi) \right|\right|_{L_\xi^2},
\end{align*}
where $C$ is a positive constant which is independent of $k$. Thus
\begin{equation}\label{H2}
    \mathrm{RHS} \,\,  \mathrm{of} \,\,  \eqref{H1} \lesssim \left( \sum_{k \in \mathbb{Z}^d} \left|\left| t^\rho e^{-t|\xi|^{2\al}} |\xi|^\eta \varphi (\xi - k) \widehat{u_0}(\xi) \right|\right|_{L_T^a L_\xi^2}^2 \right)^{\frac{1}{2}}.
\end{equation}
By using Hölder's inequality, it follows
        \begin{equation*}
            \left|\left| t^\rho e^{-t|\xi|^{2\al}} \right|\right|_{L_T^a} \leq T^{\frac{1}{a} - \frac{1}{a'}} \left|\left| t^\rho e^{-t|\xi|^{2\al}} \right|\right|_{L_T^{a'}},
        \end{equation*}
where $\frac{1}{a'} = \frac{1}{2\al} (\eta - 2\al \rho - s) \leq \frac{1}{a}$. It is obvious that $\frac{1}{a} - \frac{1}{a'} = \si$ and
\begin{align*}
    \left|\left| t^\rho e^{-t|\xi|^{2\al}} \right|\right|_{L_T^{a'}} & = \left( \int_0^T t^{a'\rho} e^{-a't|\xi|^{2\al}} \, dt \right) ^{\frac{1}{a'}} \\
        & = \left( \int_0^T |\xi|^{-2\al a' \rho} \left( t|\xi|^{2\al} \right)^{a'\rho} e^{-a't|\xi|^{2\al}} \, \frac{d\left(t|\xi|^{2\al}\right)}{|\xi|^{2\al}} \right) ^{\frac{1}{a'}} \\
        & \lesssim |\xi|^{-2\al \rho - \frac{2\al}{a'}} \left( \int_0^\infty \tau^{a' \rho} e^{-a'\tau} \, d\tau \right)^{\frac{1}{a'}} 
        \lesssim 
        |\xi|^{s-\eta},
\end{align*}
from which and \eqref{H2} one obtains
\begin{align*}
    \mathrm{RHS} \,\,  \mathrm{of} \,\,  \eqref{H1} & \lesssim T^\si \left( \sum_{k \in \mathbb{Z}^d} \int_{\mathbb{R}^d} \varphi^2(\xi - k) |\xi|^{2s} |\widehat{u_0}(\xi)|^{2} \, d\xi \right)^{\frac{1}{2}}.
\end{align*}
Note that $0\le \varphi(\xi) \leq 1$ and $\sum_{k \in \mathbb{Z}^d} \varphi(\xi - k) = 1$, we conclude \begin{equation}\label{SE-1}
    ||h^\omega||_{L_\omega^{r_s} L_{\rho; T}^a \Dot{W}_x^{\eta,p}} \lesssim T^\si \left( \int_{\mathbb{R}^d}  |\xi|^{2s} |\widehat{u_0}(\xi)|^{2} \, d\xi \right)^{\frac{1}{2}} \lesssim T^\si \, ||u_0||_{\Dot{H}^s}.
\end{equation}

(2) In the case $a = \infty$, $2\leq p \leq r_s$ and $\eta -2\al \rho \leq s$, by using Lemma \ref{PQ} we get
\begin{align*}
    ||h^\omega||_{L_{\rho; T}^\infty \Dot{W}_x^{\eta,p}} 
 & = \sup_{t \in [0,T]} t^\rho \left|\left| \Lambda^\eta e^{-t(-\Delta)^\al} \sum_{k\in \mathbb{Z}^d} g_k(\omega) \varphi (\Lambda -k) u_0 \right|\right|_{L_x^p} \\
 & = \sup_{t \in [0,T]} t^\rho \left|\left| \Lambda^{\eta-s} e^{-t(-\Delta)^\al} \sum_{k\in \mathbb{Z}^d} g_k(\omega) \varphi (\Lambda -k) \Lambda^s u_0 \right|\right|_{L_x^p} \\
 & \lesssim \sup_{t \in [0,T]} t^\rho \cdot t^{-\frac{\eta-s}{2\al}} \left|\left| \sum_{k\in \mathbb{Z}^d} g_k(\omega) \varphi (\Lambda -k) \Lambda^s u_0 \right|\right|_{L_x^p} \\
 & \lesssim T^{\frac{1}{2\al}(s-(\eta-2\al \rho))} \left|\left| \sum_{k\in \mathbb{Z}^d} g_k(\omega) \varphi (\Lambda -k) \Lambda^s u_0 \right|\right|_{L_x^p}.
\end{align*}
By using generalized Minkowski's inequality, Lemma \ref{NW} and Bernstein's inequality, one concludes
\begin{equation}\label{SE-2}
    \begin{aligned}
        ||h^\omega||_{L_\omega^{r_s} L_{\rho; T}^\infty \Dot{W}_x^{\eta,p}} & \lesssim T^\si \left|\left| \sum_{k\in \mathbb{Z}^d} g_k(\omega) \varphi (\Lambda -k) \Lambda^s u_0 \right|\right|_{L_x^p L_\om^{r_s}} \\
        & \lesssim T^\si \left( \sum_{k\in \mathbb{Z}^d} \left|\left| \varphi (\Lambda -k) \Lambda^s u_0 \right|\right|_{L_x^p}^2 \right)^{\frac{1}{2}} \\
        & \lesssim T^\si \left( \sum_{k\in \mathbb{Z}^d} \left|\left| \varphi (\Lambda -k) \Lambda^s u_0 \right|\right|_{L_x^2}^2 \right)^{\frac{1}{2}} \\
        & \lesssim T^\si \left( \sum_{k\in \mathbb{Z}^d}  \int_{\mathbb{R}^d}\varphi^2(\xi -k) |\xi|^{2s} |\widehat{u_0}(\xi)|^2 \, d\xi \right)^{\frac{1}{2}} \\
        & \lesssim T^\si \, ||u_0||_{\Dot{H}^s},
    \end{aligned}
\end{equation}
the last inequality holds since $\varphi(\xi) \leq 1$ and $\sum_{k \in \mathbb{Z}^d} \varphi(\xi - k) = 1$.

Combining \eqref{SE-1} and \eqref{SE-2} we obtain the estimate \eqref{SE1}, from which one gets that the inequality \eqref{SE2} holds by using the Bienaymé–Tchebichev inequality.

\end{proof}

For simplicity, from now we shall use the following notations:
\begin{equation}\label{DPA}
    a = \frac{4\al}{2\al(\mu+1)-1}, ~\ p = \frac{2d}{2\al(1-\mu)-1}, ~\ b = \frac{4\al}{1-2\al\mu}
\end{equation}
and
\begin{equation}\label{Q}
    q = \frac{2d}{d+1-2\al(1-\mu)},\quad \la=\frac{2d}{d+1-2\al(\mu+1)}
\end{equation}
 where $\mu$ is defined in \eqref{mu}. It is clear that
\begin{equation}\label{AB}
    \frac{1}{a} + \frac{1}{b} = \frac{1}{2}, ~\ \ \frac{1}{p} + \frac{1}{q} = \frac{1}{2},
\end{equation}
and
\begin{equation}\label{AP}
    \frac{2\al}{a} + \frac{d}{p} = 2\al - 1.
\end{equation}

For any fixed $T\in (0, +\infty]$, introduce the solution spaces,
\begin{equation}\label{Y}
    Y_T := \left\{
    \begin{aligned}
        & Y_{T,1},~\ \ \ \ \al \in (\frac{2}{3},1], s \in [-\frac{\al}{2},0) ~\ \mathrm{or} ~\ \al \in (\frac{1}{2},\frac{2}{3}], s \in (1-2\al,0), \\
        & Y_{T,2},~\ \ \ \ \al \in (\frac{2}{3},1], s \in (1-2\al,-\frac{\al}{2}), \\
        & Y_{T,3}, ~\ \ \ \ \al \in ( 1,\frac{d+2}{4} ], s \in [-1,0), \\
        & Y_{T,4}, ~\ \ \ \ \al \in ( 1,\frac{d+2}{4} ], s \in (-\al,-1),
    \end{aligned}
    \right.
\end{equation}
where $Y_{T,j}$ $(j=1,2,3,4)$ are defined by
\begin{align*}
    Y_{T,1} & := L_T^a L_x^p \cap L_T^a L_x^\la \cap L_T^a L_x^q \cap L_T^b W_x^{1-\al(2\mu+1),p} \cap L_{\frac{3\mu+1}{4}-\frac{1}{8\al};T}^{\frac{8\al}{1-2\al(\mu-1)}} W_x^{\frac{1}{2},p}, \\
    Y_{T,2} & := L_T^a L_x^p \cap L_T^a L_x^\la \cap L_T^a L_x^q \cap L_{1-\frac{1}{2\al};T}^a W_x^{2\al-1,p}, \\
    Y_{T,3} & := L_T^a L_x^p \cap L_T^a L_x^\la \cap L_T^a L_x^q \cap L_T^b L_x^p \cap L_{\frac{1}{2\al};T}^a W_x^{1,p}, \\
    Y_{T,4} & := Y_{T,3} \cap L_T^a L_x^{\frac{2d}{2\al(1-\mu)-2s-3}},
\end{align*}
endowed with the natural norms
\begin{align*}
    ||f||_{Y_{T,1}} & := ||f||_{L_T^a L_x^p} + ||f||_{L_T^a L_x^\la} + ||f||_{L_T^a L_x^q} +||f||_{L_T^b W_x^{1-\al(2\mu+1),p}} +||f||_{L_{\frac{3\mu+1}{4}-\frac{1}{8\al};T}^{\frac{8\al}{1-2\al(\mu-1)}} W_x^{\frac{1}{2},p}}, \\
    ||f||_{Y_{T,2}} & := ||f||_{L_T^a L_x^p} + ||f||_{L_T^a L_x^\la} + ||f||_{L_T^a L_x^q} +||f||_{L_{1-\frac{1}{2\al};T}^a W_x^{2\al-1,p}}, \\
    ||f||_{Y_{T,3}} & := ||f||_{L_T^a L_x^p} + ||f||_{L_T^a L_x^\la} + ||f||_{L_T^a L_x^q} + ||f||_{L_T^b L_x^p} + ||f||_{L_{\frac{1}{2\al};T}^a W_x^{1,p}}, \\
    ||f||_{Y_{T,4}} & := ||f||_{Y_{T,3}} + ||f||_{L_T^a L_x^{\frac{2d}{2\al(1-\mu)-2s-3}}}.
\end{align*}

In addition, we shall also another function spaces,
\begin{align*}
    X_{T,1} & := L_{\mu;T}^{\frac{4\al}{2\al(1-\mu)-1}} L_x^p \cap L_{\mu;T}^{\frac{4\al}{2\al(1-\mu)-1}} L_x^q, \\
    X_{T,2} & := L_{-\frac{s}{2\al};T}^\infty L_x^2 \cap L_{-\frac{s}{2\al};T}^2 H_x^\al,
\end{align*}
endowed with the natural norms
\begin{align*}
    ||f||_{X_{T,1}} & := ||f||_{L_{\mu;T}^{\frac{4\al}{2\al(1-\mu)-1}} L_x^p} + ||f||_{L_{\mu;T}^{\frac{4\al}{2\al(1-\mu)-1}} L_x^q}, \\
    ||f||_{X_{T,2}} & := ||f||_{L_{-\frac{s}{2\al};T}^\infty L_x^2} + ||f||_{L_{-\frac{s}{2\al};T}^2 H_x^\al}.
\end{align*}

Now, let us give a remark for the motivation of introducing these spaces. Inspired by \cite{NPS}, we study the problem \eqref{DE} of the difference $w=u-h^\om$ instead of $u$ to prove Theorem \ref{TW}, see also \cite{DZ,WW,LW}. Similar to \cite{ZF1}, one can construct a mild solution $w_1$ to the problem \eqref{DE} in $L_\tau^a L_x^p$ with $a$ and $p$ satisfying \eqref{AP} and $\tau > 0$ small. However, for example, when $\al \in (\frac{2}{3},1], s \in [-\frac{\al}{2},0)$ or $\al \in (\frac{1}{2},\frac{2}{3}], s \in (1-2\al,0)$, we need additional regularity, i.e. $w_1$ and $h^\om$ belong to $L_\tau^a L_x^\la \cap L_\tau^a L_x^q \cap L_\tau^b W_x^{1-\al(2\mu+1),p} \cap X_{\tau,1}$, to ensure that $w_1$ is actually a weak solution on $[0,\tau]$ in the sense of Definition \ref{DEWS}. Note that $w_1,h^\om \in L_\tau^a L_x^p$ is not enough to obtain that $w_1$ lies in $X_{\tau,1}$, thus we need an auxiliary space, i.e. $L_{\frac{3\mu+1}{4}-\frac{1}{8\al};\tau}^{\frac{8\al}{1-2\al(\mu-1)}} W_x^{\frac{1}{2},p}$, to guarantee $w_1 \in X_{\tau,1}$. In addition, the regularity $w_1 \in L_\tau^a L_x^\la \cap L_\tau^a L_x^q \cap L_\tau^b W_x^{1-\al(2\mu+1),p}$ cannot be obtained by using $w_1 \in L_\tau^a L_x^p \cap L_{\frac{3\mu+1}{4}-\frac{1}{8\al};\tau}^{\frac{8\al}{1-2\al(\mu-1)}} W_x^{\frac{1}{2},p}$, that's why we define $Y_T$ as in \eqref{Y}. Next, we shall use a fixed point argument in $Y_\tau$ to get a mild solution $w_1$ defined in a short time interval $[0,\tau]$. Moreover,   the singularity of the mild solution $w_1$ disappears as $t>0$, thus one can use the standard energy method to obtain a weak solution $w_2$ defined on $[\frac{\tau}{2},T]$ for any fixed $T>0$. Finally, in order to get a well-defined weak solution on $[0,T]$, we additionally need to estimate $h^\om$ in $X_{T,2}$. That's why we introduce the above spaces.

With the help of Lemma \ref{SE}, we have the following properties for $h^\om$, the proof is similar to that given in \cite[Proposition 2.1]{LW}, thus we omit the detail.

\begin{proposition}\label{PH}
    Assume that $\{g_k(\omega)\}_{k\in \mathbb{Z}^d}$ satisfies the condition given in \eqref{HN}, $u_0 \in \Dot{H}_\si^s(\mathbb{R}^d)$ with $d \geq 2$ and $s \in (-\al+(1-\al)_+,0)$, then for any fixed $T \in (0,+\infty)$, there exists an event $\Sigma \subset \Om$ with 
probability one, i.e. $\mathbb{P}(\Sigma) = 1$, such that for any $\om \in \Si$, the solution $h^\om$ to the problem \eqref{FE} lies in the following space:
    \begin{equation}\label{RH}
        h^\om \in Y_T \cap X_{T,1} \cap X_{T,2}.
    \end{equation}
    Moreover, when $\al \in \left( \frac{1}{2},1 \right]$, we further have
    \begin{equation}\label{RH-1}
        h^\om \in L^{\frac{4\al}{2\al \mu + 1}} \left( 0,T; W^{1-\al,p}(\mathbb{R}^d) \right).
    \end{equation}
\end{proposition}

\section{Almost sure global existence of weak solutions}

By the divergence-free condition of $u$, we rewrite the problem \eqref{GNS} with the initial data $u_0^\omega$ as
\begin{equation}\label{GNS1}
    \left\{
    \begin{aligned}
        \partial_t u + (-\Delta)^\al u + B(u,u) & = 0, \\
        u|_{t=0} & = u_0^\omega,
    \end{aligned}
    \right.
\end{equation}
where $B(f,g) = \mathrm{P} (f\cdot \nabla) g$, with $\mathrm{P}=Id+\nabla (-\Delta)^{-1} \mathrm{div}$ being the Leray projection.
Note that $\mathrm{div}\, h^\om = 0$, then from \eqref{GNS1} and \eqref{FE}, we know that $w=u-h^\om$ satisfies the following problem
\begin{equation}\label{DE}
    \left\{
    \begin{aligned}
       \partial_t w + (-\Delta)^\al w + B(w,w) + B(w,h^\om) + B(h^\om,w) + B(h^\om,h^\om) & = 0, \\
        w|_{t=0} & = 0.
    \end{aligned}
\right.
\end{equation}

We start with the definition of weak solutions to the initial value problem \eqref{DE}.

\begin{definition}\label{DEWS}
A function $w \in L^a\left( 0,T;L^\la(\mathbb{R}^d) \right)$, satisfying $t^\mu w \in L^\infty ((0,T];L_\si^2(\mathbb{R}^d)) \cap L^2((0,T];\Dot{H}_\si^\al(\mathbb{R}^d))$, is a 
 weak solution of the problem \eqref{DE} on $[0,T]$, if for any divergence-free test function $\psi \in C_c^\infty(\mathbb{R}^d)$ and a.e. $t \in [0,T]$, it holds that
    \begin{equation*}
        \langle \frac{dw}{dt},\psi \rangle + \langle \Lambda^\al w, \Lambda^\al \psi \rangle - \langle B(w,\psi), w \rangle - \langle B(w,\psi), h^\om \rangle - \langle B(h^\om,\psi), w \rangle - \langle B(h^\om,\psi), h^\om \rangle = 0,
    \end{equation*}
    and
    \begin{equation*}
        \lim_{t \rightarrow 0^+} || w(t)||_{\Dot{H}^s} = 0.
    \end{equation*}
    
    We say that the problem \eqref{DE} has a global weak solution if for any $T>0$, it admits a weak solution on $[0,T]$.
\end{definition}

The goal of this section is to establish the following result.

\begin{theorem}\label{DT}
    Let $\Si$ be the event constructed in Proposition \ref{PH}, then for any $\om \in \Si$, the problem \eqref{DE} admits a global weak solution in the sense of Definition \ref{DEWS}.
\end{theorem}

Obviously, $u=w+h^\om$ is a global weak solution to the problem \eqref{GNS} in the sense of Definition \ref{DW}, providing that $w$ is a global weak solution to the problem \eqref{DE} in the sense of Definition  \ref{DEWS}. Thus, the existence result stated in Theorem \ref{TW} follows from Theorem \ref{DT}. 

\subsection{Proof of Theorem \ref{DT} for $\al \in \left( \frac{1}{2},1 \right]$ with $d \geq 2$.}
\ 

Denote by 
\begin{equation}
    M(f,g)=\int_0^t e^{-(t-\tau)(-\Delta)^\al} B(f,g)(\tau,\cdot) \, d\tau
\end{equation}
for two divergence-free fields $f$ and $g$, for which we have the following estimate.

\begin{lemma}\label{MBL}
    For any fixed $T \in (0,1)$, assume that $m,r \in [2,\infty)$ and $\be \in [0,1]$, there exists a positive constant $C$ depending only on $m,r$ and $\be$, such that
    \begin{equation}\label{MB}
        ||M(f,g)||_{L_T^m W_x^{\be,r}} \leq C \left( ||(f,g)||_{L_T^a L_x^p}^2 + ||(f,g)||_{L_T^m W_x^{\be,r}}^2 \right),
    \end{equation}
    where $||(f,g)||_X := ||f||_X + ||g||_X$.
\end{lemma}

\begin{proof}
For any divergence-free fields $f,h \in C_c^\infty (\mathbb{R}^d)$ and $g \in C_c^\infty (\mathbb{R}^d)$, $\beta \in [0,1]$ and $l \in (1,\infty)$, by using Lemma \ref{PR} and Hölder's inequality, one has
    \begin{equation}\label{M10}
        \begin{aligned}
            \left| \int_{\mathbb{R}^d} B(f,g) \cdot h\, dx \right| & = \left| \int_{\mathbb{T}^d} (f\cdot \nabla) g \cdot h \, dx \right| \\
            & \lesssim ||\Lambda^{1-\be}(f \otimes g)||_{L_x^{l}} ||\Lambda^\be h||_{L_x^{l'}} \\
            & \lesssim \left( ||\Lambda^{1-\be} f||_{L_x^{l_1}} ||g||_{L_x^{l_2}} + ||f||_{L_x^{l_3}} ||\Lambda^{1-\be}  g||_{L_x^{l_4}} \right) ||\Lambda^\be h||_{L_x^{l'}} \\
            & \lesssim \left( ||f||_{W_x^{1-\be,l_1}} ||g||_{L_x^{l_2}} + ||f||_{L_x^{l_3}} ||g||_{W_x^{1-\be,l_4}} \right) ||h||_{W_x^{\be,l'}},
        \end{aligned}
    \end{equation}
    where $\frac{1}{l}=1-\frac{1}{l'}=\frac{1}{l_1}+ \frac{1}{l_2} = \frac{1}{l_3} + \frac{1}{l_4}$. From \eqref{M10} and Lemma A.1 in \cite{LW}, we obtain
\begin{equation}\label{M11}
    ||B(f,g)||_{W_x^{-\be,l}} \leq ||B(f,g)||_{\Dot{W}_x^{-\be,l}} \lesssim ||f||_{W_x^{1-\be,l_1}} ||g||_{L_x^{l_2}} + ||f||_{L_x^{l_3}} ||g||_{W_x^{1-\be,l_4}}.
\end{equation}
Moreover, for any $f \in L_x^{l_3} \cap W_x^{1-\be,l_1}$ and $g \in L_x^{l_2} \cap W_x^{1-\be,l_4}$, the following estimate holds
\begin{equation}\label{M12}
    ||\Lambda^{-\be} B(f,g)||_{L_x^l} \lesssim ||\Lambda^{1-\be} f||_{L_x^{l_1}} ||g||_{L_x^{l_2}} + ||f||_{L_x^{l_3}} ||\Lambda^{1-\be} g||_{L_x^{l_4}}.
\end{equation}

    Using the estimate \eqref{IH} and $T<1$, it follows
\begin{align*}
    \|M(f,g)\|_{L_T^m W_x^{\be,r}} 
 & = \left|\left|\int_0^t e^{-(t-\tau)(-\Delta)^\al} (1-\Delta)^{\frac{1}{2}}(1-\Delta)^{\frac{\be-1}{2}} B(f,g) \, d\tau \right|\right|_{L_T^m L_x^r} \\
 &  \lesssim \left|\left|\int_0^t \left( 1+(t-\tau)^{\frac{1}{2\al}} \right) (t-\tau)^{-\frac{1}{2\al} - \frac{d}{2\al} \left( \frac{1}{p} \right) } ||B(f,g)||_{W_x^{\be-1,\frac{pr}{p+r}}} \, d\tau \right|\right|_{L_T^m} \\
 & \lesssim \left( 1+T^{\frac{1}{2\al}} \right) \left|\left|\int_0^t (t-\tau)^{-\frac{1}{2\al} - \frac{d}{2\al} \left( \frac{1}{p} \right) } ||B(f,g)||_{W_x^{\be-1,\frac{pr}{p+r}}} \, d\tau \right|\right|_{L_T^m} \\
 & \lesssim \left|\left|\int_0^t (t-\tau)^{-\frac{1}{2\al} - \frac{d}{2\al} \left( \frac{1}{p} \right) } ||B(f,g)||_{W_x^{\be-1,\frac{pr}{p+r}}} \, d\tau \right|\right|_{L_T^m}.
\end{align*}
Noting that the relation \eqref{AP} between $a$ and $p$ implies
\begin{equation*}
    \frac{1}{2\al} + \frac{d}{2\al} \cdot \frac{1}{p}= 1 - \left( \frac{a+m}{am} - \frac{1}{m} \right),
\end{equation*}
by using Lemma \ref{EY}, the estimate \eqref{M11} and Hölder's inequality, one gets
\begin{align*}
    \|M(f,g)\|_{L_T^m W_x^{\be,r}} & \lesssim \Big|\Big| ||B(f,g)||_{W_x^{\be-1,\frac{pr}{p+r}}} \Big|\Big|_{L_T^\frac{am}{a+m}} \\
 & \lesssim \Big|\Big| ||f||_{L_x^p} ||g||_{W_x^{\be,r}} + ||f||_{W_x^{\be,r}} ||g||_{L_x^p}\Big|\Big|_{L_T^\frac{am}{a+m}} \\
 & \lesssim ||f||_{L_T^a L_x^p} ||g||_{L_T^m {W_x^{\be,r}}} + ||f||_{L_T^m {W_x^{\be,r}}} ||g||_{L_T^a L_x^p},
\end{align*}
from which we conclude the estimate \eqref{MB}.
\end{proof}

When $t$ near zero, we rewrite the problem \eqref{DE} as the following integral form
\begin{equation}\label{IF}
   \begin{array}{ll}
   w(t,x) = (Kw)(t,x)& :=  - \int_0^t e^{-(t-\tau)(-\Delta)^\al} \Big( B(w,w) + B(w,h^\om) + B(h^\om, w) + B(h^\om,h^\om) \Big) d\tau\\[2mm]
   &= -M(w,w)-M(w,h^\om)-M(h^\om,w)-M(h^\om,h^\om).
   \end{array}
\end{equation}

With the help of Lemma \ref{MBL}, we have

\begin{proposition}\label{CM}
    For any fixed $T \in (0,1)$, the operator $K$ defined in \eqref{IF} satisfies
    \begin{equation}\label{CM1}
        ||K(w)||_{Y_T} \leq C_K \left( ||w||_{Y_T}^2 + ||h^\omega||_{Y_T}^2 \right),
    \end{equation}
    and
    \begin{equation}\label{CM2}
        ||K(w_1)-K(w_2)||_{Y_T} \leq C_K ||w_1-w_2||_{Y_T}(||w_1||_{Y_T}+||w_2||_{Y_T}+||h^\omega||_{Y_T}),
    \end{equation}
    where $C_K$ is a positive constant and the space $Y_T$ is defined in \eqref{Y}.
\end{proposition}

\begin{proof}
We shall derive the estimate \eqref{CM1}, and \eqref{CM2} can be obtained similarly. As we consider the case $\alpha\in ({\frac 12}, 1]$ in this subsection, so we shall only need to investigate the case $Y_T=Y_{T,1}$ or $Y_T=Y_{T,2}$ given in \eqref{Y}.

\hspace{.02in}
\noindent \textbf{\underline{$\bullet$ Case 1: $\al \in (\frac{2}{3},1], s \in [-\frac{\al}{2},0) ~\ \mathrm{or} ~\ \al \in (\frac{1}{2},\frac{2}{3}], s \in (1-2\al,0)$, i.e. $Y_T=Y_{T,1}$.}}

\vspace{0.5em}

In this case, we claim that the following lemma holds, and its proof shall be given in Appendix \ref{A}.
\begin{lemma}\label{MEL}
    For any fixed $T \in (0,1)$, assume that $\be \in [0,1]$, $\ka \in \left[0, \frac{3\mu+1}{4}-\frac{1}{8\al} \right]$ with $\mu$ being given in \eqref{mu}, $m \in (2,\infty)$ and $r \in [2,\infty)$ satisfy
    \begin{equation}\label{H3.2}
        \frac{\be}{2\al} - \ka + \frac{1}{a} = \frac{1}{m},
    \end{equation}
    with $a$ being given in \eqref{DPA}.
     Then there exists a positive constant $C$, such that
     \begin{equation}\label{ME}
         ||M(f,g)||_{L_{\ka;T}^m W_x^{\be,r}} \leq C \left( ||(f,g)||_{L_T^a L_x^r}^2 + ||(f,g)||_{L_T^b W_x^{1-\al(2\mu+1),p}}^2 + ||(f,g)||_{L_{\frac{3\mu+1}{4}-\frac{1}{8\al};T}^{\frac{8\al}{1-2\al(\mu-1)}} W_x^{\frac{1}{2},p}}^2 \right).
      \end{equation}
\end{lemma}

Now, let us show the estimate \eqref{CM1} in this situation. 

On one hand, by choosing $\be=0$, $m=a$ and $r=p$ ($\la$ and $q$ respectively) in Lemma \ref{MBL}, we obtain
\begin{equation}\label{CM11}
    \begin{aligned}
        ||K(w)||_{L_T^a L_x^p} +
    ||K(w)||_{L_T^a L_x^\la} +
    ||K(w)||_{L_T^a L_x^q} 
    & \lesssim ||(w,h^\om)||_{L_T^a L_x^p}^2 + ||(w,h^\om)||_{L_T^a L_x^\la}^2 + ||(w,h^\om)||_{L_T^a L_x^q}^2 \\
    & \lesssim ||w||_{Y_{T,1}}^2 + ||h^\om||_{Y_{T,1}}^2,
    \end{aligned}
\end{equation}
by using the form \eqref{IF} of the operator $K$.

Similarly, by taking $m=b$, $r=p$ and $\be=1-\al(2\mu+1)$ in Lemma \ref{MBL}, it follows
\begin{equation}\label{CM12}
    ||K(w)||_{L_T^b W_x^{1-\al(2\mu+1),p}} \lesssim ||(w,h^\om)||_{L_T^a L_x^p}^2 + ||(w,h^\om)||_{L_T^b W_x^{1-\al(2\mu+1),p}}^2 \lesssim ||w||_{Y_{T,1}}^2 + ||h^\om||_{Y_{T,1}}^2.
\end{equation}

On the other hand, by using Lemma \ref{MEL} with $\be=\frac{1}{2}$, $\ka=\frac{3\mu+1}{4}-\frac{1}{8\al}$, $m=\frac{8\al}{1-2\al(\mu-1)}$ and $r=p$, one gets
\begin{equation}\label{CM13}
    \begin{aligned}
        ||K(w)||_{L_{\frac{3\mu+1}{4}-\frac{1}{8\al};T}^{\frac{8\al}{1-2\al(\mu-1)}} W_x^{\frac{1}{2},p}} & \lesssim ||(w,h^\om)||_{L_T^a L_x^p}^2 + ||(w,h^\om)||_{L_T^b W_x^{1-\al(2\mu+1),p}}^2 + ||(w,h^\om)||_{L_{\frac{3\mu+1}{4}-\frac{1}{8\al};T}^{\frac{8\al}{1-2\al(\mu-1)}} W_x^{\frac{1}{2},p}}^2 \\
        & \lesssim \, ||w||_{Y_{T,1}}^2 + ||h^\om||_{Y_{T,1}}^2.
    \end{aligned}
\end{equation}

Combining \eqref{CM11}-\eqref{CM13} we conclude the estimate \eqref{CM1}. 

\hspace{.02in}
\noindent \textbf{\underline{$\bullet$ Case 2: $\al \in (\frac{2}{3},1], s \in (1-2\al, -\frac{\al}{2})$, i.e. $Y_T=Y_{T,2}$.}}

\vspace{0.5em}

Alternatively, we have the following lemma in this situation, its proof shall be given in Appendix \ref{A}.
\begin{lemma}\label{MEL2}
    For any fixed $T \in (0,1)$, assume that $\be \in [0,1]$, $\ka \in \left[0, 1-\frac{1}{2\al} \right]$, $m \in [2,\infty)$ and $r \in [2,\infty)$ satisfy
    \begin{equation}\label{H3.3}
        \frac{\be}{2\al} - \ka + \frac{1}{a} = \frac{1}{m},
    \end{equation}
     with $a$ being given in \eqref{DPA}. Then there exists a positive constant $C$, such that
     \begin{equation}\label{ME2}
         ||M(f,g)||_{L_{\ka;T}^m W_x^{\be,r}} \leq C \left( ||(f,g)||_{L_T^a L_x^r}^2 + ||(f,g)||_{L_T^a L_x^p}^2 + ||(f,g)||_{L_{1-\frac{1}{2\al};T}^a W_x^{2\al-1,p}}^2 \right).
      \end{equation}
\end{lemma}

Note that $Y_T = Y_{T,2}$ at this time. By using Lemma \ref{MEL2} with $\be=2\al-1$, $\ka=1-\frac{1}{2\al}$, $m=a$ and $r=p$, one obtains
\begin{equation}\label{CM14}
    \begin{aligned}
        ||K(w)||_{L_{1-\frac{1}{2\al};T}^a W_x^{2\al-1,p}} 
        & \lesssim ||(w,h^\om)||_{L_T^a L_x^p}^2 + ||(w,h^\om)||_{L_{1-\frac{1}{2\al};T}^a W_x^{2\al-1,p}}^2 \\
        & \lesssim ||w||_{Y_{T,2}}^2 + ||h^\om||_{Y_{T,2}}^2.
    \end{aligned}
\end{equation}

Adding \eqref{CM11} and \eqref{CM14} together we conclude the estimate \eqref{CM1} in this case. The estimate \eqref{CM2} can be obtained in the same way as above. 

\end{proof}

\begin{proposition}\label{MS}
    Let $\Si$ be the event constructed in Proposition \ref{PH}, then for any $\om \in \Si$, there exist $\tau_\om \in (0,1)$ and a mild solution\footnote{We say that $w_1$ is a mild solution to the problem \eqref{DE} if $w_1$ is a solution of the integral equation \eqref{IF}.} $w_1$ to the problem \eqref{DE} on the time interval $[0,\tau_\om]$. Moreover,
    \begin{equation}\label{RM}
        w_1 \in Y_{\tau_\om} \cap X_{{\tau_\om},1} \cap L_{\tau_\om}^{\frac{4\al}{2\al \mu + 1}} W_x^{1-\al,p} \cap L_{\mu;\tau_\om}^2 \Dot{H}_x^\al.
    \end{equation}
\end{proposition}

\begin{proof}
    From \eqref{RH}, we know there is $\tau_\om \in (0,1)$ small, such that
\begin{equation*}
    \la_{\tau_\om} := ||h^\om||_{Y_{\tau_\om}} \leq \frac{1}{4C_K},
\end{equation*}
where $C_K$ is the constant given on the right hand side of \eqref{CM1}-\eqref{CM2}. It is easy to see from Proposition \ref{CM} that the operator $K$ defined by \eqref{IF} maps the closed subset 
$$B:=\left\{ w \in (Y_{\tau_\om})_\si : ||w||_{Y_{\tau_\om}} \leq 2C_K \lambda_{\tau_\om}^2 \right\}$$
into itself, with $(Y_{\tau_\om})_\si$ denoting the divergence-free subspace of $Y_{\tau_\om}$, and for any two $w$ and $\widetilde{w}$ in 
the set $B$, one has
$$\|K(w)-K(\widetilde{w})\|_{Y_{\tau_\om}}\le \frac{1}{2}\|w-\widetilde{w}\|_{Y_{\tau_\om}}.$$
So, there is a fixed point $w_1$ of $K$ on the closed subset $B$, which is a mild solution of the problem \eqref{DE} on $(Y_{\tau_\om})_\si$.

Next, let us verify that the mild solution $w_1$ satisfies the property \eqref{RM}, which shall also be studied for two cases respectively.

\hspace{.02in}
\noindent \textbf{\underline{$\bullet$ Case 1: $\al \in (\frac{2}{3},1], s \in [-\frac{\al}{2},0) ~\ \mathrm{or} ~\ \al \in (\frac{1}{2},\frac{2}{3}], s \in (1-2\al,0)$.}}

\vspace{0.5em}

In this case we know $Y_{\tau_\om} = Y_{\tau_\om,1}$. 

By taking $\be=0$, $\ka=\mu$, $m=\frac{4\al}{2\al(1-\mu)-1}$ and $r=p$ ($q$ respectively) in Lemma \ref{MEL}, we have
$$||K(w)||_{L_{\mu;\tau_\om}^{\frac{4\al}{2\al(1-\mu)-1}} L_x^p} + ||K(w)||_{L_{\mu;\tau_\om}^{\frac{4\al}{2\al(1-\mu)-1}} L_x^q} \lesssim ||w||_{Y_{\tau_\om,1}}^2 + ||h^\om||_{Y_{\tau_\om,1}}^2,$$
by noting the form \eqref{IF} of the operator $K$. Thus, one gets
\begin{equation}\label{RM1}
    w_1=K(w_1)\in X_{\tau_\om,1},
\end{equation}
by using \eqref{RH} and $w_1 \in Y_{\tau_\om,1}$.

Similarly, by choosing $\be=1-\al$, $\ka=0$, $m=\frac{4\al}{2\al\mu+1}$ and $r=p$ in Lemma \ref{MEL}, it follows
$$||K(w_1)||_{L_{\tau_\om}^{\frac{4\al}{2\al\mu+1}} W_x^{1-\al,p}} \lesssim ||w||_{Y_{\tau_\om,1}}^2 + ||h^\om||_{Y_{\tau_\om,1}}^2,$$
from which and \eqref{RH} we obtain
\begin{equation}\label{RM2}
    w_1=K(w_1)\in L_{\tau_\om}^{\frac{4\al}{2\al \mu + 1}} W_x^{1-\al,p}.
\end{equation}

To show $w_1 \in L_{\mu;\tau_\om}^2 \Dot{H}_x^\al$, decompose $M(f,g)$ into 
 \begin{equation}\label{ME11}
 M(f,g)=I_1+I_2:=
 \int_0^\frac{t}{2}e^{-(t-\tau)(-\Delta)^\al} B(f,g) \, d\tau + \int_\frac{t}{2}^t e^{-(t-\tau)(-\Delta)^\al} B(f,g) \, d\tau.
 \end{equation}
On one hand, by using Corollary \ref{MR1} with $\zeta = 2\al(\mu+1)$, and the inequality \eqref{M11},
we get that the term $I_1$ defined in \eqref{ME11} satisfies
\begin{equation}\label{RM11}
    \begin{aligned}
        \|I_1\|_{L_{\mu;{\tau_\om}}^2 \Dot{H}_x^{\al}} 
        & =  \left|\left|\int_0^\frac{t}{2} t^\mu e^{-(t-\tau)(-\Delta)^\al} \Lambda^{2\al(\mu+1)} \Lambda^{-\al(2\mu+1)} B(f,g) \, d\tau \right|\right|_{L_{\tau_\om}^2 L_x^2} \\
        & \lesssim  ||\Lambda^{-\al(2\mu+1)} B(f,g)||_{L_{\tau_\om}^2 L_x^2}  \\
        & \lesssim ||f||_{L_{\tau_\om}^a L_x^q} ||g||_{L_{\tau_\om}^b W_x^{1-\al(2\mu+1),p}} + ||f||_{L_{\tau_\om}^b W_x^{1-\al(2\mu+1), p}} ||g||_{L_{\tau_\om}^a L_x^q} \\
        & \lesssim ||f||_{Y_{\tau_\om,1}} ||g||_{Y_{\tau_\om,1}}.
    \end{aligned}
\end{equation}

On the other hand, to estimate the term $I_2$ given in \eqref{ME11}, by using Lemma \ref{PQ}, one obtains
\begin{align*}
     ||I_2||_{L_{\mu;\tau_\om}^2 \Dot{H}_x^\al} & = \left|\left| \int_\frac{t}{2}^t t^\mu e^{-(t-\tau)(-\Delta)^\al} \Lambda^{\al+\frac{1}{2}} \Lambda^{-\frac{1}{2}} B(f,g) \, d\tau \right|\right|_{L_{\tau_\om}^2 L_x^2} \\
     & \lesssim \left|\left| \int_\frac{t}{2}^t t^\mu (t-\tau)^{-\frac{\al+\frac{1}{2}}{2\al}} ||\Lambda^{-\frac{1}{2}} B(f,g)||_{L_x^2} \, d\tau \right|\right|_{L_{\tau_\om}^2} \\
     & \lesssim \left|\left| \int_\frac{t}{2}^t t^{\mu-\left(\frac{3\mu+1}{4}-\frac{1}{8\al}\right)} (t-\tau)^{-\frac{\al+\frac{1}{2}}{2\al}} \tau^{\frac{3\mu+1}{4}-\frac{1}{8\al}} ||\Lambda^{-\frac{1}{2}}B(f,g)||_{L_x^2} \, d\tau \right|\right|_{L_{\tau_\om}^2} \\
     & \lesssim \left|\left| \int_\frac{t}{2}^t (t-\tau)^{\mu-\left(\frac{3\mu+1}{4}-\frac{1}{8\al}\right)-\frac{\al+\frac{1}{2}}{2\al}} \tau^{\frac{3\mu+1}{4}-\frac{1}{8\al}} ||\Lambda^{-\frac{1}{2}} B(f,g)||_{L_x^2} \, d\tau \right|\right|_{L_{\tau_\om}^2},
 \end{align*}
 by noting that $\mu-(\frac{3\mu +1}{4}-\frac{1}{8\alpha})\le 0$.
 
 From the definition of $a$ given in \eqref{DPA}, one has the identity 
 $$-\mu+\left(\frac{3\mu+1}{4}-\frac{1}{8\al}\right)+\frac{\al+\frac{1}{2}}{2\al} = 1-\left( \frac{1}{a} + \frac{1-2\al(\mu-1)}{8\al} - \frac{1}{2} \right),$$
 so by using Lemma \ref{EY} and the estimate \eqref{M11}, it follows
 \begin{align*}
     ||I_2||_{L_{\mu;\tau_\om}^2 \Dot{H}_x^\al} & \lesssim \Big|\Big| \tau^{\frac{3\mu+1}{4}-\frac{1}{8\al}} ||\Lambda^{-\frac{1}{2}} B(f,g)||_{L_x^2} \Big|\Big|_{L_{\tau_\om}^m} \\
     & \lesssim \Big|\Big| \tau^{\frac{3\mu+1}{4}-\frac{1}{8\al}} \left( ||f||_{L_x^q} ||g||_{W_x^{\frac{1}{2},p}}  + ||g||_{L_x^q} ||f||_{W_x^{\frac{1}{2},p}} \right) \Big|\Big|_{L_{\tau_\om}^m},
 \end{align*}
 where $\frac{1}{m}=\frac{1}{a} + \frac{1-2\al(\mu-1)}{8\al}$. Thus, by using Hölder's inequality, one gets
\begin{align*}
    ||I_2||_{L_{\mu;\tau_\om}^2 \Dot{H}_x^\al} & \lesssim ||f||_{L_{\tau_\om}^a L_x^q} ||g||_{L_{\frac{3\mu+1}{4}-\frac{1}{8\al};\tau_\om}^{\frac{8\al}{1-2\al(\mu-1)}} W_x^{\frac{1}{2},p}} + ||g||_{L_{\tau_\om}^a L_x^q} ||f||_{L_{\frac{3\mu+1}{4}-\frac{1}{8\al};\tau_\om}^{\frac{8\al}{1-2\al(\mu-1)}} W_x^{\frac{1}{2},p}} \\
         & \lesssim ||f||_{Y_{\tau_\om,1}} ||g||_{Y_{\tau_\om,1}},
\end{align*}
which together with the estimate \eqref{RM11} implies
$$||K(w_1)||_{L_{\mu;\tau_\om}^2 \Dot{H}_x^\al} \lesssim ||w_1||_{Y_{\tau_\om,1}}^2 +  ||h^\om||_{Y_{\tau_\om,1}}^2,$$
by using the form \eqref{IF} of the operator $K$. Since $w_1,h^\om \in Y_{\tau_\om,1}$, one has
\begin{equation}\label{RM12}
    w_1=K(w_1) \in L_{\mu;\tau_\om}^2 \Dot{H}_x^\al.
\end{equation}

Combining \eqref{RM1}-\eqref{RM2} and \eqref{RM12}, we conclude the regularity \eqref{RM} in this case.

\hspace{.02in}
\noindent \textbf{\underline{$\bullet$ Case 2: $\al \in (\frac{2}{3},1], s \in (1-2\al, -\frac{\al}{2})$.}}

\vspace{0.5em}

In this case one has $Y_{\tau_\om} = Y_{\tau_\om,2}$.

By taking $\be=0$, $\ka=\mu$, $m=\frac{4\al}{2\al(1-\mu)-1}$ and $r=p$ ($q$ respectively) in Lemma \ref{MEL2}, we have
$$||K(w)||_{L_{\mu;\tau_\om}^{\frac{4\al}{2\al(1-\mu)-1}} L_x^p} + ||K(w)||_{L_{\mu;\tau_\om}^{\frac{4\al}{2\al(1-\mu)-1}} L_x^q} \lesssim ||w||_{Y_{\tau_\om,2}}^2 + ||h^\om||_{Y_{\tau_\om,2}}^2,$$
by noting the form \eqref{IF} of the operator $K$. Thus, by using \eqref{RH} and $w_1 \in Y_{\tau_\om,2}$, one gets
\begin{equation}\label{RM3}
    w_1=K(w_1)\in X_{\tau_\om,1}.
\end{equation}

Similarly, by choosing $\be=1-\al$, $\ka=0$, $m=\frac{4\al}{2\al\mu+1}$ and $r=p$ in Lemma \ref{MEL}, it follows
$$||K(w_1)||_{L_{\tau_\om}^{\frac{4\al}{2\al\mu+1}} W_x^{1-\al,p}} \lesssim ||w||_{Y_{\tau_\om,2}}^2 + ||h^\om||_{Y_{\tau_\om,2}}^2,$$
from which and \eqref{RH} we obtain
\begin{equation}\label{RM4}
    w_1=K(w_1)\in L_{\tau_\om}^{\frac{4\al}{2\al \mu + 1}} W_x^{1-\al,p}.
\end{equation}

To show $w_1 \in L_{\mu,\tau_\om}^2 \Dot{H}_x^\al$, on one hand, for the term $I_1$ defined in \eqref{ME11}, by using Lemma \ref{PQ}, one has
\begin{align*}
    \|I_1\|_{L_{\mu;{\tau_\om}}^2 \Dot{H}_x^{\al}} 
        & =  \left|\left|\int_0^\frac{t}{2} t^\mu e^{-(t-\tau)(-\Delta)^\al} \Lambda^{1+\al} \Lambda^{-1} B(f,g) \, d\tau \right|\right|_{L_{\tau_\om}^2 L_x^2} \\
        & \lesssim \left|\left|\int_0^\frac{t}{2} t^\mu (t-\tau)^{-\frac{1+\al}{2\al}} ||\Lambda^{-1} B(f,g)||_{L_x^2} \, d\tau \right|\right|_{L_{\tau_\om}^2} \\
        & \lesssim \left|\left|\int_0^\frac{t}{2} (t-\tau)^{\mu-\frac{1+\al}{2\al}} ||f||_{L_x^p} ||g||_{L_x^q} \, d\tau \right|\right|_{L_{\tau_\om}^2},
\end{align*}
by using the estimate \eqref{M11}. Moreover, by noting
$$-\mu+\frac{1+\al}{2\al}=1-\left( \frac{2}{a}-\frac{1}{2} \right),$$
and using Lemma \ref{EY}, it holds
\begin{equation}\label{RM13}
    \begin{aligned}
        \|I_1\|_{L_{\mu;{\tau_\om}}^2 \Dot{H}_x^{\al}} \lesssim \Big|\Big| ||f||_{L_x^p} ||g||_{L_x^q} \Big|\Big|_{L_{\tau_\om}^{\frac{a}{2}}} \lesssim ||f||_{L_{\tau_\om}^a L_x^p} ||g||_{L_{\tau_\om}^a L_x^q} \lesssim ||f||_{Y_{\tau_\om,2}} ||g||_{Y_{\tau_\om,2}}.
    \end{aligned}
\end{equation}

On the other hand, by using Lemma \ref{PQ} one gets that the term $I_2$ defined in \eqref{IF} satisfies
\begin{align*}
     ||I_2||_{L_{\mu;\tau_\om}^2 \Dot{H}_x^\al} & = \left|\left| \int_\frac{t}{2}^t t^\mu e^{-(t-\tau)(-\Delta)^\al} \Lambda^{2-\al} \Lambda^{2\al-2} B(f,g) \, d\tau \right|\right|_{L_{\tau_\om}^2 L_x^2} \\
     & \lesssim \left|\left| \int_\frac{t}{2}^t t^\mu (t-\tau)^{-\frac{2-\al}{2\al}} ||\Lambda^{2\al-2} B(f,g)||_{L_x^2} \, d\tau \right|\right|_{L_{\tau_\om}^2} \\
     & \lesssim \left|\left| \int_\frac{t}{2}^t t^{\mu-1+\frac{1}{2\al}} (t-\tau)^{-\frac{2-\al}{2\al}} \tau^{1-\frac{1}{2\al}}||\Lambda^{2\al-2} B(f,g)||_{L_x^2} \, d\tau \right|\right|_{L_{\tau_\om}^2} \\
     & \lesssim \left|\left| \int_\frac{t}{2}^t (t-\tau)^{\mu-1+\frac{1}{2\al}-\frac{2-\al}{2\al}} \tau^{1-\frac{1}{2\al}}||\Lambda^{2\al-2} B(f,g)||_{L_x^2} \, d\tau \right|\right|_{L_{\tau_\om}^2},
 \end{align*}
 by noting $\mu-1+\frac{1}{2\alpha}\le 0$.
 
From the definition of $a$ given in \eqref{DPA}, one has 
$$-\mu+1-\frac{1}{2\al}+\frac{2-\al}{2\al} = 1-\left( \frac{2}{a} - \frac{1}{2} \right),$$
so by using Lemma \ref{EY} and the estimate \eqref{M11}, we have
\begin{equation}\label{I2}
    \begin{aligned}
        ||I_2||_{L_{\mu;\tau_\om}^2 \Dot{H}_x^\al} & \lesssim \Big|\Big| \tau^{1-\frac{1}{2\al}}||\Lambda^{2\al-2} B(f,g)||_{L_x^2} \Big|\Big|_{L_{\tau_\om}^{\frac{a}{2}}} \\
    & \lesssim \Big|\Big| 
    \tau^{1-\frac{1}{2\al}} \left( ||f||_{L_x^q} ||g||_{W_x^{2\al-1,p}} + ||g||_{L_x^q} ||f||_{W_x^{2\al-1,p}} \right) \Big|\Big|_{L_{\tau_\om}^{\frac{a}{2}}} \\
    & \lesssim ||f||_{L_{\tau_\om}^a L_x^q} ||g||_{L_{1-\frac{1}{2\al};\tau_\om}^a W_x^{2\al-1,p}} + ||g||_{L_{\tau_\om}^a L_x^q} ||f||_{L_{1-\frac{1}{2\al};\tau_\om}^a W_x^{2\al-1,p}} \\
    & \lesssim ||f||_{Y_{\tau_\om,2}} ||g||_{Y_{\tau_\om,2}},
    \end{aligned}
\end{equation}
from which and the inequality \eqref{RM13}, one obtains
$$||K(w_1)||_{L_{\mu;\tau_\om}^2 \Dot{H}_x^\al} \lesssim ||w_1||_{Y_{\tau_\om,2}}^2 +  ||h^\om||_{Y_{\tau_\om,2}}^2,$$
by noting the form \eqref{IF} of the operator $K$. Using the fact that $w_1$ and $h^\om$ belong to $Y_{\tau_\om,2}$, we get
\begin{equation}\label{RM14}
    w_1 = K(w_1) \in L_{\tau_\om}^2 \Dot{H}_x^\al.
\end{equation}

Combining \eqref{RM3}-\eqref{RM4} and \eqref{RM14}, we obtain the regularity of $w_1$ given in \eqref{RM}.

\end{proof}

\begin{proposition}\label{P1}
    The mild solution $w_1$ constructed in Proposition \ref{MS} satisfies
    \begin{equation}\label{P11}
        w_1 \in C\left( [0,{\tau_\om}]; \Dot{H}^{s}(\mathbb{R}^d) \right).
    \end{equation}
\end{proposition}

\begin{proof}
    It is enough to show that
    \begin{equation}\label{P12}
        w_1 \in L^\infty \left( 0,{\tau_\om}; \Dot{H}^s(\mathbb{R}^d) \right).
    \end{equation}
If \eqref{P12} is true, then the continuity given in \eqref{P11} can be obtained in a way similar to that given in the second part of the proof of Proposition 3.4 given in \cite{LW}.

Now, let us verify the property \eqref{P12}. By using Lemma \ref{PQ}, the estimate \eqref{M11} and Hölder's inequality, we get that for any $t \in (0,\tau_\om]$,
\begin{equation}\label{U-1}
    \begin{aligned}
        ||M(f,g)||_{\Dot{H}_x^s} & = \left|\left| \int_0^t e^{-(t-\tau)(-\Delta)^\al} \Lambda^s B(f,g) \, d\tau \right|\right|_{L_x^2}  \lesssim \int_0^t (t-\tau)^{-\frac{1+s}{2\al}} ||B(f,g)||_{\Dot{H}_x^{-1}} \, d\tau \\
    & \lesssim \int_0^t \tau^{-\mu} (t-\tau)^{-\frac{1+s}{2\al}} \tau^{\mu} \Big( ||f||_{L_x^p} ||g||_{L_x^q} \Big) \, d\tau \\ 
    & \lesssim \left( \int_0^t \left( \tau^{-\mu} (t-\tau)^{-\frac{1+s}{2\al}} \right)^{2\al} d\tau \right)^\frac{1}{2\al} ||f||_{L_{\tau_\om}^a L_x^p} ||g||_{L_{\mu,\tau_\om}^{\frac{4\al}{2\al(1-\mu)-1}} L_x^q}.
    \end{aligned}
\end{equation}
Note that $s\in (1-2\al,0)$ in this case, one obtains
\begin{align*}
    \int_0^t \left( \tau^{-\mu} (t-\tau)^{-\frac{1+s}{2\al}} \right)^{2\al} d\tau
    = C t^{-s-2\al \mu} = 
    \left\{
    \begin{aligned}
       C t^{s-1+2\al} \leq C, \quad & s\in \left( 1-2\al,\frac{1-2\al}{2} \right), \\
       Ct^{-s} \leq C, \quad & s\in \left[ \frac{1-2\al}{2},0 \right),
    \end{aligned}
    \right. 
\end{align*}
which together with \eqref{U-1} implies
$$||M(f,g)||_{\Dot{H}_x^s} \lesssim ||f||_{L_{\tau_\om}^a L_x^p} ||g||_{L_{\mu,\tau_\om}^{\frac{4\al}{2\al(1-\mu)-1}} L_x^q} \lesssim ||f||_{Y_{\tau_\om}} ||g||_{X_{\tau_\om,1}}.$$
Thus, we conclude the claim \eqref{P12} from the form \eqref{IF} of the operator $K$ and $w,h^\om \in Y_{\tau_\om} \cap X_{\tau_\om,1}$.

\end{proof}

\begin{proposition}\label{P2}
    The mild solution $w_1$ constructed in Proposition \ref{MS} satisfies
    \begin{equation}\label{P21}
        t^\mu w_1 \in C\left([0,\tau_\om];L^2(\mathbb{R}^d) \right).
    \end{equation}
\end{proposition}
\begin{proof}
    The proof is similar to that given in Proposition \ref{P1}. We only shall show that
    \begin{equation}\label{P22}
        t^\mu w_1 \in L^\infty \left( 0,{\tau_\om}; L^2(\mathbb{R}^d) \right),
    \end{equation}
    from which we shall obtain \eqref{P21} by performing the same argument as in the second part of the proof of Proposition 3.4 in \cite{LW}.

Now, let us verify the property \eqref{P22} in the following two situations.

\hspace{.02in}
\noindent \textbf{\underline{$\bullet$ Case 1: $s \in \left( 1-2\al,\frac{1-2\al}{2} \right)$.}}

\vspace{0.5em}

In this case one has $\mu>0$, so by using Lemma \ref{PQ}, the estimate \eqref{M11} and Hölder's inequality, we obtain that for any $t \in (0,\tau_\om]$,
\begin{align*}
    ||t^\mu M(f,g)||_{L_x^2} & = \left|\left| \int_0^t t^\mu e^{-(t-\tau)(-\Delta)^\al} B(f,g) \, d\tau \right|\right|_{L_x^2}  \lesssim \int_0^t t^\mu (t-\tau)^{-\frac{1}{2\al}} ||\Lambda^{-1} B(f,g)||_{L_x^2} \, d\tau \\
    & \lesssim \int_0^t \Big( (t-\tau)^\mu + \tau^\mu \Big) \tau^{-2\mu} (t-\tau)^{-\frac{1}{2\al}} \tau^{2\mu} \Big( ||f||_{L_x^p} ||g||_{L_x^q} \Big) \, d\tau \\ 
    & \lesssim \left( \int_0^t \left\{ \Big( (t-\tau)^\mu + \tau^\mu \Big) \tau^{-2\mu} (t-\tau)^{-\frac{1}{2\al}} \right\}^m d\tau \right)^\frac{1}{m} ||f||_{L_{\mu;\tau_\om}^{\frac{4\al}{2\al(1-\mu)-1}} L_x^p} ||g||_{L_{\mu;\tau_\om}^{\frac{4\al}{2\al(1-\mu)-1}} L_x^q},
\end{align*}
where $\frac{1}{m} = \mu + \frac{1}{2\al}$. Note that
$$\int_0^t \left\{ \Big( (t-\tau)^\mu + \tau^\mu \Big) \tau^{-2\mu} (t-\tau)^{-\frac{1}{2\al}} \right\}^m d\tau \leq C,$$
it follows
$$||t^\mu M(f,g)||_{L_x^2} \lesssim ||f||_{L_{\mu;\tau_\om}^{\frac{4\al}{2\al(1-\mu)-1}} L_x^p} ||g||_{L_{\mu;\tau_\om}^{\frac{4\al}{2\al(1-\mu)-1}} L_x^q} \lesssim ||f||_{X_{\tau_\om,1}} ||g||_{X_{\tau_\om,1}},$$
which implies the claim \eqref{P22} from the form \eqref{IF} of the operator $K$ and $w_1,h^\om \in X_{\tau_\om,1}$.

\hspace{.02in}
\noindent \textbf{\underline{$\bullet$ Case 2: $s \in \left[ \frac{1-2\al}{2},0 \right)$.}}

\vspace{0.5em}

In this case one has $\mu=0$ and $Y_T=Y_{T,1}$. Using Lemma \ref{ML}, the estimate \eqref{M11}, and Hölder's inequality, one gets that for any $t \in (0,\tau_\om]$,
\begin{align*}
    ||M(f,g)||_{L_x^2} & = \left|\left| \int_0^t e^{-(t-\tau)(-\Delta)^\al} \Lambda^{\al} \Lambda^{-\al} B(f,g) \, d\tau \right|\right|_{L_x^2} \lesssim ||\Lambda^{-\al} B(f,g)||_{L_T^2 L_x^2} \\
    & \lesssim ||f||_{L_T^b W_x^{1-\al,p}} ||g||_{L_T^a L_x^q} + ||f||_{L_T^a L_x^q} ||g||_{L_T^b W_x^{1-\al,p}} \lesssim ||f||_{Y_{\tau_\om,1}} ||g||_{Y_{\tau_\om,1}}.
\end{align*}
Thus, we conclude the regularity \eqref{P22} by using the definition \eqref{IF} of the operator $K$ and $w_1,h^\om \in Y_{\tau_\om,1}$.
    
\end{proof}

\noindent \textbf{Proof of Theorem \ref{DT} for $\al \in \left( \frac{1}{2},1 \right]$ with $d \geq 2$:}

\vspace{0.5em}

The problem \eqref{DE} shall be studied for $t$ near zero and  away from zero separately.

    (1) When $t$ near zero. 

From Proposition \ref{MS} we have a mild solution $w_1$ to the problem \eqref{DE} on $[0,\tau_\om]$. Now, let us show that $w_1$ is a weak solution on $[0,\tau_\om]$ of the problem \eqref{DE} in the sense of Definition \ref{DEWS}.

Similar to Lemma 11.3 given in \cite{LR}, from the definition of $K(w)$ given in \eqref{IF}, one gets that 
\begin{equation}\label{WS1}
    \partial_t K(w_1) = -(-\Delta)^\al K(w_1) - B(w_1,w_1) - B(w_1,h^\om) - B(h^\om,w_1) - B(h^\om,h^\om),
\end{equation}
in the sense of $\mathcal{D}'(\mathbb{R}^d)$ and a.e. $t \in [0,\tau_\om]$. Moreover, due to $w_1=K(w_1)$, we obtain that $w_1$ satisfies the following equation
\begin{equation}\label{WS2}
    \partial_t w_1 = -(-\Delta)^\al w_1 - B(w_1,w_1) - B(w_1,h^\om) - B(h^\om,w_1) - B(h^\om,h^\om),
\end{equation}
in $\mathcal{D}'(\mathbb{R}^d)$ and a.e. $t \in [0,\tau_\om]$. Meanwhile, Propositions \ref{MS}-\ref{P2} imply
$$t^\mu w_1 \in L^\infty \left(0,{\tau_\om}; L_\si^2(\mathbb{R}^d)\right) \cap L^2\left( 0,{\tau_\om}; \Dot{H}_\si^\al(\mathbb{R}^d) \right),$$
$$ w_1 \in C\left( [0,{\tau_\om}]; \Dot{H}^s(\mathbb{R}^d) \right) \cap L^a \left( 0,\tau_\om; L^\la(\mathbb{R}^d) \right),$$
and
$$\lim_{t \rightarrow 0^+ } ||w_1(t)||_{\Dot{H}_x^s} = 0.$$
Therefore, $w_1$ is a weak solution of the problem \eqref{DE} in the sense of Definition \ref{DEWS} on $[0,{\tau_\om}]$.

(2) When $t$ away from zero.

Note that Proposition \ref{P2} implies
$$||w_1\left(\frac{{\tau_\om}}{2}\right)||_{L_x^2} \lesssim \left( \frac{{\tau_\om}}{2} \right)^{-\mu} ||w_1||_{L_{\mu; {\tau_\om}}^\infty L_x^2} \leq C_{\tau_\om}.$$
Now, take $w_1\left(\frac{{\tau_\om}}{2} \right)$ as a new initial data of the problem \eqref{DE} and use energy method on $\left[\frac{{\tau_\om}}{2}, T\right]$.

Applying the standard energy estimate in the problem \eqref{DE}, we have
$$\frac{d}{dt}||w||_{L_x^2}^2 + ||\Lambda^\al w||_{L_x^2}^2 \leq C \left( ||w||_{L_x^2}^2 ||\Lambda^{1-\al} h^\om||_{L_x^p}^{\frac{4\al}{2\al \mu+1}} + ||\Lambda^{1-\al} h^\om||_{L_x^q}^2 ||h^\om||_{L_x^p}^2 \right),$$
from which and Grönwall's inequality, one obtains that for any given $T>0$ and $t\in [\frac{\tau_\om}{2}, T]$, it follows
\begin{equation}\label{WS8}
    E_w(t) \leq C \left(  || w\left( \frac{{\tau_\om}}{2} \right) ||_{L_x^2}^2 + \int_{\frac{{\tau_\om}}{2}}^T ||\Lambda^{1-\al} h^\om||_{L_x^q}^2 ||h^\om||_{L_x^p}^2 \, d\tau \right) \exp{\left( \int_{\frac{{\tau_\om}}{2}}^T ||\Lambda^{1-\al} h^\om||_{L_x^p}^{\frac{4\al}{2\al \mu+1}} d\tau \right)},
\end{equation}
where
$$E_w(t)= ||w(t)||_{L_x^2}^2 + \int_{\frac{{\tau_\om}}{2}}^t ||\Lambda^\al w||_{L_x^2} d\tau.$$

Using Lemma \ref{PQ} and $\frac{1}{p}+\frac{1}{q}=\frac{1}{2}$, one has
$$||\Lambda^{1-\al} h^\om||_{L_x^q}^2 ||h^\om||_{L_x^p}^2 \lesssim t^{-\frac{1-\al-2s}{\al}-\frac{d}{2\al}},
$$
from which and $\frac{1-\al-2s}{\al}+\frac{d}{2\al}>1$, we get
\begin{equation}\label{WS9}
    \begin{aligned}
        \int_{\frac{{\tau_\om}}{2}}^T ||\Lambda^{1-\al} h^\om||_{L_x^q}^2 ||h^\om||_{L_x^p}^2 d\tau 
        \lesssim \int_{\frac{{\tau_\om}}{2}}^T \tau^{-\frac{1-\al-2s}{\al}-\frac{d}{2\al}} d\tau \leq C_{\tau_\om}.
    \end{aligned}
\end{equation}
Similarly, one obtains
\begin{equation}\label{WS10}
    \int_{\frac{{\tau_\om}}{2}}^T ||\Lambda^{1-\al} h^\om||_{L_x^p}^{\frac{4\al}{2\al \mu+1}} d\tau \lesssim \int_{\frac{{\tau_\om}}{2}}^T \tau^{\left\{ -\frac{1-\al-s}{2\al}-\frac{d}{2\al}\left( \frac{1}{2}-\frac{1}{p} \right) \right\}\frac{4\al}{2\al \mu+1}} d\tau \leq C_{\tau_\om},
\end{equation}
by using 
$$\left\{ \frac{1-\al-s}{2\al}+\frac{d}{2\al}\left( \frac{1}{2}-\frac{1}{p} \right) \right\}\frac{4\al}{2\al \mu+1} > 1.$$

Plugging \eqref{WS9}-\eqref{WS10} into the right hand side of \eqref{WS8}, we obtain the following a priori bound 
\begin{equation}\label{WS16}
    E_w(t) \leq C_{\tau_\om}, \quad \forall t \in \left[ \frac{\tau_\om}{2},T \right].
\end{equation}

Next, we estimate the $L^2 \left( \frac{\tau_\om}{2},T; H_x^{-\frac{3d}{4}} \right)$ norm of $\partial_t w$. For any $\psi \in L^2 \left(\frac{\tau_\om}{2},T;H_x^{\frac{3d}{4}}\right)$, by using Hölder's inequality, Gagliardo–Nirenberg inequality and the Sobolev embedding $H_x^\al \hookrightarrow L_x^{\frac{2d}{d-2\al}}$, we have
\begin{align*}
    \left| \int_{\frac{\tau_\om}{2}}^T \langle B(f,g),\psi \rangle \, d\tau \right| & = \left| \int_{\frac{\tau_\om}{2}}^T \langle (f \cdot \nabla) \psi,g \rangle \, d\tau \right| \lesssim \int_{\frac{\tau_\om}{2}}^T ||f||_{L_x^2} ||\nabla \psi||_{L_x^{\frac{d}{\al}}} ||g||_{L_x^{\frac{2d}{d-2\al}}}  \, d\tau \\
        & \lesssim \int_{\frac{\tau_\om}{2}}^T ||f||_{L_x^2} || \psi||_{L_x^2}^{\frac{d-4(1-\al)}{3d}} || \psi||_{H_x^{\frac{3d}{4}}}^{\frac{2d+4(1-\al)}{3d}} ||g||_{H_x^\al}  \, d\tau \\
        & \lesssim \int_{\frac{\tau_\om}{2}}^T ||f||_{L_x^2} || \psi||_{H_x^{\frac{3d}{4}}} ||g||_{H_x^\al}  \, d\tau \\
        & \lesssim ||f||_{L^\infty(\frac{\tau_\om}{2},T;L_x^2)} ||g||_{L^2(\frac{\tau_\om}{2},T;H_x^\al)} ||\psi||_{L^2(\frac{\tau_\om}{2},T;H_x^{\frac{3d}{4}})},
\end{align*}
which implies
\begin{equation}\label{WS-1}
    ||B(f,g)||_{L^2(\frac{\tau_\om}{2},T;H_x^{-\frac{3d}{4}})} \lesssim ||f||_{L^\infty(\frac{\tau_\om}{2},T;L_x^2)} ||g||_{L^2(\frac{\tau_\om}{2},T;H_x^\al)}.
\end{equation}
Similarly, since
\begin{align*}
    \left| \int_{\frac{\tau_\om}{2}}^T \langle (-\Delta)^\al w, \psi \rangle \, d\tau \right| \leq \int_{\frac{\tau_\om}{2}}^T ||\Lambda^\al w||_{L_x^2} ||\Lambda^\al \psi||_{L_x^2} \, d\tau 
    \lesssim ||w||_{L^2(\frac{\tau_\om}{2},T;H_x^\al)} ||\psi||_{L^2(\frac{\tau_\om}{2},T;H_x^{\frac{3d}{4}})},
\end{align*}
one obtains
\begin{equation}\label{WS-2}
    ||(-\Delta)^\al w||_{L^2(\frac{\tau_\om}{2},T;H_x^{-\frac{3d}{4}})} \lesssim ||w||_{L^2(\frac{\tau_\om}{2},T;H_x^\al)}.
\end{equation}

Combining \eqref{WS2} and \eqref{WS-1}-\eqref{WS-2}, we conclude
\begin{equation}\label{WS-3}
    \begin{aligned}
        ||\partial_t w||_{L^2(\frac{\tau_\om}{2},T;H_x^{-\frac{3d}{4}})} \lesssim ||w||_{L^2(\frac{\tau_\om}{2},T;H_x^\al)} + ||(w,h^\om)||_{L^\infty(\frac{\tau_\om}{2},T;L_x^2)} ||(w,h^\om)||_{L^2(\frac{\tau_\om}{2},T;H_x^\al)} \lesssim C_{\tau_\om},
    \end{aligned}
\end{equation}
by using the property \eqref{RH} and the bound \eqref{WS16}.

   Next, we sketch the main idea to get the existence of a solution to the problem \eqref{DE} in $[\frac{\tau_\om}{2}, T]$ with $w_1\left(\frac{{\tau_\om}}{2} \right)$ as a new initial data, by using the standard Galerkin method. First, via the Galerkin approach, we can construct a sequence of approximate solutions $\{w^N\}_{N\in \mathbb{N}}$ of  \eqref{DE} with the data being $w_1\left(\frac{{\tau_\om}}{2} \right)$, satisfying estimates similar to those given in \eqref{WS16} and \eqref{WS-3}, so by using the Aubin-Lions compactness theorem, cf. \cite[pp.102-106]{BF}, one gets a subsequence $\{w^{N_k}\}_{k \in \mathbb{N}}$ and $w_2$, such that
   \begin{equation}\label{L1}
       w^{N_k} \rightarrow w_2, \quad {\rm in}\quad
    L^2\left(\frac{{\tau_\om}}{2},T;L_{loc}^2(\mathbb{R}^d) \right),
   \end{equation}
    and
    \begin{equation}\label{L2}
        w^{N_k} \rightharpoonup w_2, \quad \mathrm{in} \quad L^\infty \left( \frac{\tau_\om}{2},T; L^2(\mathbb{R}^d) \right) \cap L^2 \left( \frac{\tau_\om}{2},T; H^\al(\mathbb{R}^d) \right).
    \end{equation}
One can verify that the limit $w_2$ is a weak solution of the problem \eqref{DE} on $[\frac{\tau_\om}{2}, T]$ by using the convergences \eqref{L1}-\eqref{L2}. Therefore, we obtain a solution $w_2$ defined on $\left[\frac{{\tau_\om}}{2},T\right]$, and satisfying
    \begin{equation}\label{L3}
        w_2 \in C_{weak} \left(\left[\frac{{\tau_\om}}{2},T\right];L_\si^2(\mathbb{R}^d) \right) \cap L^2\left(\frac{{\tau_\om}}{2},T;H_\si^\al(\mathbb{R}^d) \right),
    \end{equation}
    $$\lim_{t\rightarrow \frac{{\tau_\om}}{2}+0} ||w_2(t)-w_1(\frac{{\tau_\om}}{2})||_{L_x^2} = 0,$$
    and the energy inequality
   $$||w_2(t)||_{L_x^2}^2 + 2\int_{\frac{{\tau_\om}}{2}}^t ||\Lambda^\al w_2|| d\tau \leq ||w_1(\frac{{\tau_\om}}{2})||_{L_x^2}^2 + 2\int_{\frac{{\tau_\om}}{2}}^t \langle B(w_2,w_2), h^\om \rangle + \langle B(h^\om,w_2),h^\om \rangle d\tau,$$
    where $t \in \left[ \frac{{\tau_\om}}{2},T \right]$.
    Moreover, by using the regularity \eqref{L3} and the embedding
    $$L^\infty \left( \frac{\tau_\om}{2},T; L^2(\mathbb{R}^d) \right) \cap L^2 \left( \frac{\tau_\om}{2},T; H^\al(\mathbb{R}^d)  \right) \hookrightarrow L^a \left( \frac{\tau_\om}{2},T; L^\la(\mathbb{R}^d) \right),$$
    one additionally gets
    $$w_2 \in L^a \left( \frac{\tau_\om}{2},T; L^\la(\mathbb{R}^d) \right).$$

By using Lemma \ref{UR} we know that $w_1$ coincides with $w_2$ on $\left[ \frac{{\tau_\om}}{2},{\tau_\om} \right]$. More precisely, from \eqref{RM} and Proposition \ref{P2}, we have
$$w_1 \in L^\infty \left( \frac{{\tau_\om}}{2},{\tau_\om}; L_x^2 \right) \cap L^2 \left( \frac{{\tau_\om}}{2},{\tau_\om}; H_x^\al \right) \cap L^{\frac{4\al}{2\al \mu+1}} \left( \frac{{\tau_\om}}{2},{\tau_\om}; W_x^{1-\al,p} \right),$$
    and
    $$\lim_{t\rightarrow \frac{{\tau_\om}}{2}+0} ||w_1(t)-w_1(\frac{{\tau_\om}}{2})||_{L_x^2} = 0.$$
On the other hand, from \eqref{RH}-\eqref{RH-1} one obtains
    $$h^\om  \in X_{\tau_\om,2} \cap L_{\tau_\om}^{\frac{4\al}{2\al \mu+1}} W_x^{1-\al,p} \hookrightarrow L^\infty \left( \frac{{\tau_\om}}{2},{\tau_\om}; L_x^2 \right) \cap L^2 \left( \frac{{\tau_\om}}{2},{\tau_\om}; H_x^\al \right) \cap L^{\frac{4\al}{2\al \mu+1}} \left( \frac{{\tau_\om}}{2},{\tau_\om}; W_x^{1-\al,p} \right).$$
Hence, by using Lemma \ref{UR} on $\left[ \frac{{\tau_\om}}{2}, {\tau_\om} \right]$ with $\be=1-\al$, $r=p$ and $m=\frac{4\al}{2\al \mu+1}$, we get that $w_1=w_2$ on $\left[ \frac{{\tau_\om}}{2}, {\tau_\om} \right]$.
    
    Thus, through defining
    \begin{equation}\label{WS20}
        w(t)=\left\{
        \begin{aligned}
        & w_1(t),\,\,\,\,\,\,\,\,\,t\in [0,{\tau_\om}],\\
        & w_2(t),\,\,\,\,\,\,\,\,\,t\in [{\tau_\om},T],
        \end{aligned}
        \right.
    \end{equation}
    it is obvious that $w(t)$ is a weak solution to the initial value problem \eqref{DE} on $[0,T]$ in the sense of Definition \ref{DEWS}, the proof is completed.

    \qed

\subsection{Sketch of proof of Theorem \ref{DT} for $\al \in \left( 1,\frac{d+2}{4} \right]$ with $d \geq 3$.}
~\

Similar to Lemmas \ref{MEL}-\ref{MEL2}, one gets
\begin{equation}\label{S-2}
    ||K(w)||_{L_{\frac{1}{2\al};T}^a W_x^{1,p}} \lesssim ||(w,h^\om)||_{L_T^a L_x^p}^2 + ||(w,h^\om)||_{L_{\frac{1}{2\al};T}^a W_x^{1,p}}^2,
\end{equation}
from which and Lemma \ref{MBL} we know that the estimates \eqref{CM11}-\eqref{CM12} hold in this case, where $Y_T=Y_{T,3}$, and $Y_{T,4}$ respectively, is defined in \eqref{Y}. So by using the fixed point argument, see Proposition \ref{MS} for detail, we can obtain a mild solution $w_1$ defined on $[0,\tau_\om]$, which belongs to $Y_{\tau_\om}$ and satisfies
\begin{equation}\label{S-3}
    w_1 \in L_{\mu;\tau_\om}^{\frac{4\al}{2\al(1-\mu)-1}} L_x^q \cap L_{\mu;\tau_\om}^2 \Dot{H}_x^\al.
\end{equation}
To show the property \eqref{S-3}, similar to Lemmas \ref{MEL}-\ref{MEL2}, we find
\begin{align*}
    ||K(w_1)||_{L_{\tau_\om}^{\frac{4\al}{2\al(1-\mu)-1}} L_x^q} & \lesssim ||(w_1,h^\om)||_{L_{\tau_\om}^a L_x^p}^2 + ||(w_1,h^\om)||_{L_{\tau_\om}^a L_x^q}^2 + ||(w_1,h^\om)||_{L_{\frac{1}{2\al};\tau_\om}^a W_x^{1,p}}^2 \\
        & \lesssim ||w_1||_{Y_{\tau_\om}}^2 + ||h^\om||_{Y_{\tau_\om}}^2,
\end{align*}
which together with $w_1,h^\om \in Y_{\tau_\om}$ implies $w_1 = K(w_1) \in L_{\mu;\tau_\om}^{\frac{4\al}{2\al(1-\mu)-1}} L_x^q$. Moreover, for the term $I_2$ given in \eqref{ME11}, similar to \eqref{I2} one obtains
$$||I_2||_{L_{\mu;\tau_\om}^2 \Dot{H}_x^\al} \lesssim ||(f,g)||_{L_{\tau_\om}^a L_x^q}^2 + ||(f,g)||_{L_{\frac{1}{2\al};\tau_\om}^a W_x^{1,p}}^2 \lesssim ||f||_{Y_{\tau_\om}}^2 + ||g||_{Y_{\tau_\om}}^2,$$
from which and \eqref{RM13} we get that $w_1 = K(w_1) \in L_{\mu;\tau_\om}^2 \Dot{H}_x^\al$, by using the form \eqref{IF} of the operator $K$ and $w_1,h^\om \in Y_{\tau_\om}$.

The calculation given in \eqref{U-1} shows that the property \eqref{P12} holds as $s \in [-1,0)$.
In addition, when $s \in (-\al,-1)$, i.e. $Y_T = Y_{T,4}$, by using Lemma \ref{PQ}, the Sobolev embedding $L_x^{\frac{2d}{d-2(s+1)}} \hookrightarrow \Dot{H}_x^{s+1}$, the estimate \eqref{M11} and Hölder's inequality, we have that for any $t \in (0,\tau_\om]$,
\begin{equation}\label{S-4}
    \begin{aligned}
        ||M(f,g)||_{\Dot{H}_x^s} & = \left|\left| \int_0^t e^{-(t-\tau)(-\Delta)^\al} \Lambda^s B(f,g) \, d\tau \right|\right|_{L_x^2}  \lesssim \int_0^t  ||\Lambda^s B(f,g)||_{L_x^2} \, d\tau \\
    & \lesssim \int_0^t  ||\Lambda^{-1} B(f,g)||_{L_x^{\frac{2d}{d-2(s+1)}}} \, d\tau \lesssim \int_0^t \tau^{-\mu} \tau^{\mu} \Big( ||f||_{L_x^{\frac{2d}{2\al(1-\mu)-2s-3}}} ||g||_{L_x^q} \Big) \, d\tau \\ 
    & \lesssim \left( \int_0^t \tau^{-2\al \mu} d\tau \right)^\frac{1}{2\al} ||f||_{L_{\tau_\om}^a L_x^{\frac{2d}{2\al(1-\mu)-2s-3}}} ||g||_{L_{\mu;\tau_\om}^{\frac{4\al}{2\al(1-\mu)-1}} L_x^q} \\
    & \lesssim ||f||_{Y_{\tau_\om,4}} ||g||_{Y_{\tau_\om,4}},
    \end{aligned}
\end{equation}
which together with $w_1,h^\om \in Y_{\tau_\om,4}$ and the form \eqref{IF} of the operator $K$  implies the property \eqref{P12} also holds as $s \in (-\al,-1)$.

Moreover, by using Lemma \ref{PQ}, Hölder's inequality and the relation \eqref{AB}, for any $t \in (0,\tau_\om]$ it follows
\begin{equation}\label{S-5}
    \begin{aligned}
        ||K(w_1)||_{L_x^2} \lesssim \left( \int_0^t (t-\tau)^{-\frac{1}{\al}} d\tau \right)^{\frac{1}{2}} \Big( ||(w_1,h^\om)||_{L_T^a L_x^q}^2 + ||(w_1,h^\om)||_{L_T^b L_x^p}^2 \Big) \lesssim ||w_1||_{Y_{\tau_\om}}^2 + ||h^\om||_{Y_{\tau_\om}}^2,
    \end{aligned}
\end{equation}
the last inequality holds since $\al>1$. Combining \eqref{S-4}-\eqref{S-5} and using the argument as in the second part of the proof of Proposition 3.4 in \cite{LW}, we conclude
\begin{equation}\label{S-6}
    w_1 \in C \left( [0,\tau_\om]; L^2(\mathbb{R}^d) \right) \cap C \left( [0,\tau_\om]; \Dot{H}^s(\mathbb{R}^d) \right).
\end{equation}

The argument above ensures that $w_1$ is a weak solution on $[0,\tau_\om]$ in the sense of Definition \ref{DEWS}. Similarly, noting form \eqref{S-6} we know that $w_1(\frac{\tau_\om}{2})$ belongs to $L^2(\mathbb{R}^d)$. Thus, starting from $\frac{\tau_\om}{2}$ and using the standard energy method one can obtain a weak solution $w_2$ defined on $\left[ \frac{\tau_\om}{2},T \right]$. By using the uniqueness result given in Lemma \ref{UR} with $\be=0$, $m=a$ and $r=p$, we get that $w_1$ coincides with $w_2$ on $\left[ \frac{\tau_\om}{2},\tau_\om \right]$. Therefore, the function defined by
\begin{equation}\label{3.54}
        w(t)=\left\{
        \begin{aligned}
        & w_1(t),\,\,\,\,\,\,\,\,\,t\in [0,{\tau_\om}],\\
        & w_2(t),\,\,\,\,\,\,\,\,\,t\in [{\tau_\om},T],
        \end{aligned}
        \right.
    \end{equation}
ia a weak solution of the problem \eqref{DE}, which completes the proof of Theorem \ref{DT} in this case.

\section{The uniqueness of weak solutions for $\al=\frac{d+2}{4}$}

The aim of this section is to study the uniqueness of weak solutions to the problem \eqref{GNS} when $\al=\frac{d+2}{4}$ with $d \geq 2$. Clearly, the uniqueness of weak solutions to the problem \eqref{DE} is equivalent to the uniqueness of weak solutions to the problem \eqref{GNS}. Thus, it is enough to study the uniqueness of weak solutions to the problem \eqref{DE}, for which we have

\begin{theorem}\label{UW}
    Let $\al=\frac{d+2}{4}$ with $d \geq 2$, assume that $w$ and $\widetilde{w}$ are two weak solutions to the problem \eqref{DE} in the sense of Definition \ref{DEWS}, then $w=\widetilde{w}$.
\end{theorem}

\begin{proof}
    By the definition of $p$ and $\la$, c.f. \eqref{DPA}-\eqref{Q}, one obtains that $\la=p$ when $\al=\frac{d+2}{4}$. We consider the following two cases separately.

    \hspace{.02in}
    \noindent \textbf{\underline{$\bullet$ Case 1: $s \in \left[ \frac{1-2\al}{2},0 \right)$.}}

    \vspace{0.5em}

    In this case one gets $\mu=0$. From Definition \ref{DEWS} we have
    $$w,\widetilde{w} \in L^\infty ( 0,T; L^2(\mathbb{R}^d) ) \cap L^2( 0,T; H^\al(\mathbb{R}^d) \cap L^a( 0,T; L^p(\mathbb{R}^d) ).$$
    So, by using Corollary \ref{UR-1} with $\be=0$, $m=a$ and $r=p$, we immediately obtain the uniqueness of weak solutions of the problem \eqref{DE}.
    
    \hspace{.02in}
    \noindent \textbf{\underline{$\bullet$ Case 2: $s \in \left( -\al,\frac{1-2\al}{2} \right)$.}}

    \vspace{0.5em}

    In this case one has $\mu=-\frac{s}{\al}+\frac{1}{2\al}-1>0$. Thus, we cannot use Corollary \ref{UR-1} directly since there has a singularity $t^{-\mu}$ when $t=0$.

    For any $0<t<1$, similar to \eqref{CM1}-\eqref{CM2} one obtains that the operator $K$ defined in \eqref{IF} satisfies
    \begin{equation}\label{U2}
        ||K(w)||_{L_t^a L_x^p} \leq C_0 \left( ||w||_{L_t^a L_x^p}^2 + ||h^\om||_{L_t^a L_x^p}^2 \right),
    \end{equation}
    and
    \begin{equation}\label{U3}
        ||K(w)-K(\widetilde{w})||_{L_t^a L_x^p} \leq C_0 ||w-\widetilde{w}||_{L_t^a L_x^p} \left( ||w||_{L_t^a L_x^p} + ||\widetilde{w}||_{L_t^a L_x^p} + ||h^\om||_{L_t^a L_x^p} \right),
    \end{equation}
    where $C_0$ is a positive constant.

    Since $w,\widetilde{w},h^\om \in L_T^a L_x^p$, we can choose $\tau \in (0,1]$ small enough, such that
    $$C_0 \left( ||w||_{L_\tau^a L_x^p} + ||\widetilde{w}||_{L_\tau^a L_x^p} + ||h^\om||_{L_\tau^a L_x^p} \right) \leq \frac{1}{2},$$
    where $C_0$ is the constant given on the right hand side of \eqref{U3}.
    
    Similar to \eqref{WS1} we have
    $$\partial_t K(w) = -(-\Delta)^\al K(w) - B(w,w) - B(w,h^\om) - B(h^\om,w) - B(h^\om,h^\om),$$
    in the sense of $\mathcal{D}'(\mathbb{R}^d)$ and a.e. $t \in [0,\tau]$. Thus, one gets
    \begin{equation}\label{U-10}
        \partial_t (w-K(w)) = -(-\Delta)^\al (w-K(w)), \quad \mathrm{in} \,\, \mathcal{D}'(\mathbb{R}^d) \,\, \mathrm{and} \,\, \mathrm{a.e.} \,\, t \in [0,\tau],
    \end{equation}
    by using the fact that $w$ is a weak solution.

    Moreover, using $w \in L_T^a L_x^p$, $t^\mu w \in L_T^\infty L_x^2$ and the interpolation inequality
    $$||w||_{L_{\mu;\tau}^{\frac{4\al}{2\al(1-\mu)-1}} L_x^q} \lesssim ||w||_{L_{\mu;\tau}^\infty L_x^2}^{\frac{4\al \mu}{2\al(\mu+1)-1}} ||w||_{L_\tau^a L_x^p}^{\frac{2\al(1-\mu)-1}{2\al(\mu+1)-1}},$$
    we know that $w$ belongs to $L_{\mu;\tau}^{\frac{4\al}{2\al(1-\mu)-1}} L_x^q$. So by performing the same calculation as in \eqref{U-1} and \eqref{S-4}, and using $h^\om \in Y_\tau$, one obtains
    $$K(w) \in C\left( [0,\tau];\Dot{H}^s(\mathbb{R}^d) \right).$$
    Therefore, we get
    \begin{equation}\label{U-11}
        \lim_{t \rightarrow 0^+} ||w-K(w)||_{\Dot{H}_x^s} \leq \lim_{t \rightarrow 0^+} \left( ||w||_{\Dot{H}_x^s} + ||K(w)||_{\Dot{H}_x^s} \right) = 0.
    \end{equation}

    Let $\rho(x)$ be the standard mollifier and $\rho_\epsilon=\frac{1}{\epsilon^d}\rho(\frac{x}{\epsilon})$ with $\epsilon>0$. Define
    $$w_\epsilon (t,x) = w(t,\cdot)*\rho_\epsilon (x) \quad \mathrm{and} \quad     K_\epsilon (w)(t,x) = K(w)(t,\cdot)*\rho_\epsilon (x).$$
    From \eqref{U2} and $w,h^\om \in L_T^a L_x^p$, we know that $K(w)$ belongs to $L_\tau^a L_x^p$, it is obvious that $w_\epsilon, K_\epsilon (w) \in C^\infty(\mathbb{R}^d)$ and
    \begin{equation}\label{U1}
        w_\epsilon \rightarrow w \quad \mathrm{and} \quad K_\epsilon (w) \rightarrow K(w) \quad \mathrm{in} \quad L_\tau^a L_x^p.
    \end{equation}
    
    It follows from \eqref{U-10}-\eqref{U-11} that
    \begin{equation}\label{U4}
        \left\{
        \begin{aligned}
            \partial_t (w_\epsilon-K_\epsilon (w)) & = -(-\Delta)^\al (w_\epsilon-K_\epsilon (w)), \quad \mathrm{for} \,\, \mathrm{a.e.} \,\, t \in [0,\tau], \\
            \lim_{t \rightarrow 0^+} & ||w_\epsilon - K_\epsilon (w)||_{L_x^2} = 0.
        \end{aligned}
        \right.
    \end{equation}
    Using the uniqueness result to the linear problem \eqref{U4}, we conclude that $w_\epsilon = K_\epsilon (w)$ for a.e. $(t,x) \in (0,\tau] \times \mathbb{R}^d$, thus they equal in $L_\tau^a L_x^p$.

    For any $\epsilon>0$, using the triangular inequality, one obtains
    $$||w-K(w)||_{L_\tau^a L_x^p} 
            \leq ||w-w_\epsilon||_{L_\tau^a L_x^p} + ||w_\epsilon - K_\epsilon (w)||_{L_\tau^a L_x^p} + ||K_\epsilon (w)-K(w)||_{L_\tau^a L_x^p},$$
    which yields $w=K(w)$ in $L_\tau^a L_x^p$, by passing $\epsilon \rightarrow 0^+$ and using the convergence \eqref{U1}.

    Similarly, we have $\widetilde{w} = K(\widetilde{w})$ in $L_\tau ^a L_x^p$. Using the inequality \eqref{U3}, one gets
    \begin{align*}
        ||w-\widetilde{w}||_{L_\tau^a L_x^p} & = ||K(w)-K(\widetilde{w})||_{L_\tau^aL_x^p} \\
            & \leq C_0 ||w-\widetilde{w}||_{L_\tau^a L_x^p} \left( ||w||_{L_\tau^a L_x^p} + ||\widetilde{w}||_{L_\tau^a L_x^p} + ||h^\om||_{L_\tau^a L_x^p} \right) \\
            & \leq \frac{1}{2} ||w-\widetilde{w}||_{L_\tau^a L_x^p},
    \end{align*}
    by using the definition of $\tau$. Thus, we conclude $w=\widetilde{w}$ for a.e. $t,x \in [0,\tau] \times \mathbb{R}^d$.

    Choosing $\si \in (0,\tau)$ such that $||w(\si)-\widetilde{w}(\si)||_{L_x^2}=0$, moreover, from Definition \ref{DEWS} we find
    $$w,\widetilde{w} \in L^\infty ( \si,T; L^2(\mathbb{R}^d) ) \cap L^2( \si,T; H^\al(\mathbb{R}^d) \cap L^a( \si,T; L^p(\mathbb{R}^d) ).$$
    Therefore, by using Corollary \ref{UR-1} with $\be=0$, $m=a$ and $r=p$, one obtains $w=\widetilde{w}$ for $t \in [\si,T]$ and a.e. $x \in \mathbb{R}^d$. The proof is completed.

\end{proof}

\section{Optimal time decays for $u$ and $w$}

In this section, we shall use the modified Fourier splitting method to obtain the optimal time decay for the $L^2$ norm of the difference $w=u-h^\om$, which is constructed in \eqref{WS20} and \eqref{3.54} respectively, to this end, we need an appropriate bound for the time derivative of $||w(t)||_{L_x^2}^2$ and a $L^\infty$ estimate for $\hat{w}(t,\xi)$, given in the following two lemmas.

\begin{lemma}\label{TL}
    Let $w$ be the weak solution constructed in Theorem \ref{DT}, then there exist $T_0>e$ and a constant $C>0$ such that
    \begin{equation}\label{T}
        \frac{d}{dt}||w(t)||_{L_x^2}^2 \leq -||\Lambda^\al w(t)||_{L_x^2}^2 + Ct^{-\frac{d+2}{2\al}+1+\frac{2s}{\al}}, ~\ \ \ \ \forall t>T_0.
    \end{equation}
\end{lemma}

\begin{proof}
    We only show the inequality \eqref{T} for $\al \in \left( \frac{1}{2},1 \right]$ with $d \geq 2$. The estimate \eqref{T} for $\al \in \left( 1,\frac{d+2}{4} \right]$ with $d \geq 3$ can be obtained similarly.
    
    Applying the standard energy estimate in the problem \eqref{DE}, we get
\begin{equation}\label{T1}
    \frac{1}{2}\frac{d}{dt}||w||_{L_x^2}^2 + ||\Lambda^\al w||_{L_x^2}^2 = \langle B(w,w), h^\om \rangle + \langle B(h^\om,w), h^\om \rangle.
\end{equation}
By using Hölder's inequality, the estimate \eqref{M12} and the Sobolev embedding $\Dot{H}_x^\al \hookrightarrow L_x^{\frac{2d}{d-\al}}$, it follows
\begin{equation}\label{T2}
    \begin{aligned}
        & |\langle B(w,w), h^\om \rangle| \lesssim  ||\Lambda^{\al-1} B(w,w) ||_{L_x^{\frac{d}{d-\al}}} ||\Lambda^{1-\al} h^\om||_{L_x^{\frac{d}{\al}}} \\
        & \hspace{.2in} \lesssim  ||w||_{L_x^{\frac{2d}{d-\al}}} ||\Lambda^\al w||_{L_x^2}  ||\Lambda^{1-\al} h^\om||_{L_x^{\frac{d}{\al}}} \lesssim ||\Lambda^\al w||_{L_x^2}^2 ||\Lambda^{1-\al} h^\om||_{L_x^{\frac{d}{\al}}}.
    \end{aligned}
\end{equation}
Using Lemma \ref{PQ}, one has
\begin{equation}\label{T3}
    \begin{aligned}
        & ||\Lambda^{1-\al} h^\om||_{L_x^{\frac{d}{\al}}} = ||\Lambda^{1-\al-s} e^{-t(-\Delta)^\al} \Lambda^s u_0^\om||_{L_x^{\frac{d}{\al}}} \\ 
        & \hspace{.2in} \lesssim t^{-\frac{1-\al-s}{2\al}-\frac{d}{2\al}\left( \frac{1}{2}-\frac{\al}{d} \right)} ||u_0^\om||_{\Dot{H}_x^s} \leq C_1 t^{-\frac{1-\al-s}{2\al}-\frac{d}{2\al}\left( \frac{1}{2}-\frac{\al}{d} \right)}.
    \end{aligned}
\end{equation}
Since $-\frac{1-\al-s}{2\al}-\frac{d}{2\al}\left( \frac{1}{2}-\frac{\al}{d} \right)<0$, we can choose $T_0>e$ large enough such that
\begin{equation}\label{T4}
    C_1 T_0^{-\frac{1-\al-s}{2\al}-\frac{d}{2\al}\left( \frac{1}{2}-\frac{\al}{d} \right)} \leq \frac{1}{4},
\end{equation}
where $C_1$ is the constant given on the right hand side of \eqref{T3}. Combining \eqref{T2}-\eqref{T4} we obtain
\begin{equation}\label{T5}
    |\langle B(w,w), h^\om \rangle| \leq \frac{1}{4} ||\Lambda^\al w||_{L_x^2}^2, \quad \forall t>T_0.
\end{equation}

On the other hand, by using Hölder's inequality and the estimate \eqref{M12}, one has
\begin{equation}\label{T6}
    \begin{aligned}
        |\langle B(h^\om,w), h^\om \rangle| & \lesssim ||\Lambda^\al w||_{L_x^2} ||\Lambda^{-\al} B(h^\om,h^\om)||_{L_x^2} \lesssim ||\Lambda^\al w||_{L_x^2} ||\Lambda^{1-\al} h^\om||_{L_x^4} ||h^\om||_{L_x^4} \\
        & \leq \frac{1}{4} ||\Lambda^\al w||_{L_x^2}^2 + C||\Lambda^{1-\al} h^\om||_{L_x^4}^2 ||h^\om||_{L_x^4}^2.
    \end{aligned}
\end{equation}
In the same calculation as in \eqref{T3}, it implies
$$||\Lambda^{1-\al} h^\om||_{L_x^4} \lesssim t^{-\frac{1-\al-s}{2\al}-\frac{d}{8\al}} ~\ \ \ \ \mathrm{and} ~\ \ \ \ ||h^\om||_{L_x^4} \lesssim t^{\frac{s}{2\al}-\frac{d}{8\al}},$$
from which and \eqref{T6} we have
\begin{equation}\label{T8}
    |\langle B(h^\om,w), h^\om \rangle| \leq \frac{1}{4} ||\Lambda^\al w||_{L_x^2}^2 + C t^{-\frac{d+2}{2\al}+1+\frac{2s}{\al}}.
\end{equation}

Plugging \eqref{T5} and \eqref{T8} into the right hand side of \eqref{T1} we conclude \eqref{T}.

\end{proof}

\begin{lemma}\label{PL}
    Let $w$ be the weak solution constructed in Theorem \ref{DT}, then there exists a constant $C>0$ such that
    \begin{equation}\label{PE}
        |\hat{w}(t,\xi)| \leq C |\xi| \left( \int_{T_0}^t ||w(\tau)||_{L_x^2}^2 d\tau + t^{1+\frac{s}{\al}} \right), \quad \forall t>T_0, \, \forall \xi \in \mathbb{R}^d,
    \end{equation}
    where $\hat{w}(t,\xi)$ denotes the partial Fourier transform of $w(t,x)$ in $x$.
\end{lemma}

\begin{proof}
    It is obvious that $w$ satisfies
    \begin{equation}\label{w}
    \partial_t w + (-\Delta)^\al w = -B(w,w)-B(w,h^\om)-B(h^\om,w)-B(h^\om,h^\om).
\end{equation}
Taking the partial Fourier transform in $x\in \mathbb{R}^d$ on both sides of \eqref{w}, we obtain that $\hat{w}(t,\xi)$ satisfies
$$\partial_t \hat{w} + |\xi|^{2\al} \hat{w} = F(t,\xi) := - \mathcal{F}_{x\to \xi} \left\{ B(w,w)+B(w,h^\om)+B(h^\om,w)+B(h^\om,h^\om) \right\}$$
and $\hat{w}(0,\xi)=0$.

For any divergence-free fields $f$ and $g$, by using the properties of Fourier transform $\mathcal{F}$ and the Leray projection $\mathrm{P}$, and Hölder's inequality, one gets
$$\left| \mathcal{F} \left\{ B(f,g) \right\} \right| = \left| \mathcal{F} \left\{ \nabla \cdot \mathrm{P} (f \otimes g) \right\} \right| \lesssim |\xi| \cdot |\mathcal{F} \mathrm{P} (f \otimes g)|  \lesssim |\xi| \cdot ||(f \otimes g)||_{L_x^1} \lesssim |\xi| \cdot ||f||_{L_x^2} ||g||_{L_x^2},$$
which implies
\begin{align*}
    |F(t,\xi)| & \lesssim |\xi| \left( ||w||_{L_x^2}^2 +  ||h^\om||_{L_x^2}^2 \right).
\end{align*}
Using Lemma \ref{PQ}, we have
\begin{equation}\label{DH}
    ||h^\om||_{L_x^2} = ||e^{-t(-\Delta)^\al}u_0^\om||_{L_x^2} = ||\Lambda^{-s} e^{-t(-\Delta)^\al} \Lambda^s u_0^\om||_{L_x^2} \lesssim t^{\frac{s}{2\al}} ||u_0^\om||_{\Dot{H}_x^s} \lesssim t^{\frac{s}{2\al}}.
\end{equation}

Thus, for any $t>0$, it follows
\begin{equation}\label{w13}
    \begin{aligned}
        |\hat{w}(t,\xi)| & \lesssim \int_0^t e^{-(t-\tau)|\xi|^{2\al}} |F(t,\xi)| \, d\tau \lesssim \int_0^t e^{-(t-\tau)|\xi|^{2\al}} |\xi| \left( ||w||_{L_x^2}^2 + \tau^{\frac{s}{\al}} \right) d\tau \\
        & \lesssim |\xi| \left( \int_0^t ||w(\tau)||_{L_x^2}^2 d\tau + \int_0^t \tau^{\frac{s}{\al}} d\tau \right):= |\xi| \left( J_1+J_2 \right).
    \end{aligned}
\end{equation}
Clearly, by using $s>-\al+(1-\al)_+ \geq -\al$, one obtains
\begin{equation}\label{J2}
    J_2 = \int_0^t \tau^{\frac{s}{\al}} d\tau = C t^{1+\frac{s}{\al}}.
\end{equation}
In addition, for the term $J_1$ given in \eqref{w13}, noting from the regularity \eqref{RH1} that
\begin{equation*}
    ||w(t)||_{L_x^2} \leq \left\{
    \begin{aligned}
        Ct^{-\mu}, \quad 0 \leq & t \leq T_0, \\
        C,\quad \quad \quad & t>T_0,
    \end{aligned}
    \right.
\end{equation*}
by using $\mu \in \left[ 0,\frac{1}{2} \right)$, we know that for any $t>T_0$,
\begin{equation}\label{J1}
    \begin{aligned}
        J_1 & = \int_0^t ||w(\tau)||_{L_x^2}^2 d\tau = \int_0^{T_0} ||w(\tau)||_{L_x^2}^2 d\tau + \int_{T_0}^t ||w(\tau)||_{L_x^2}^2 d\tau \\
        & \lesssim \int_0^{T_0} \tau^{-2\mu} d\tau + \int_{T_0}^t ||w(\tau)||_{L_x^2}^2 d\tau \lesssim 1 + \int_{T_0}^t ||w(\tau)||_{L_x^2}^2 d\tau.
    \end{aligned}
\end{equation}

Combining \eqref{w13}-\eqref{J1}, we conclude the estimate \eqref{PE} by noting $1 \lesssim t^{1+\frac{s}{\al}}$ for any $t>T_0$.

\end{proof}

Inspired by \cite{W}, we have the following lemma, the proof of which shall be given in Appendix \ref{A} for the reader's convenience.

\begin{lemma}\label{L}
    For any fixed $t_0 \geq 0$, let $g(t),x(t),y(t)$ be three continuous functions defined on $[t_0,\infty)$, and $y(t)$ be strictly positive and $x(t),y(t)$ be non-negative respectively. Assume that for any $\tau_2 > \tau_1 \geq t_0$, it holds
    \begin{equation}\label{L-1}
        x(\tau_2) + \int_{\tau_1}^{\tau_2} g^{2\al}(\tau) x(\tau) d\tau \leq x(\tau_1) + \int_{\tau_1}^{\tau_2} y(\tau) d\tau.
    \end{equation}
    Then, one has
    \begin{equation}\label{L-7}
        x(t) e^{\int_{t_0}^t g^{2\al}(\tau) d\tau} \leq x(t_0) + \int_{t_0}^t e^{ \int_{t_0}^\tau g^{2\al}(\tau') d\tau' } y(\tau) d\tau,
    \end{equation}
    for any $t>t_0$.
\end{lemma}

With Lemmas \ref{TL}-\ref{L} in hand, we shall prove Theorem \ref{TD}.

\vspace{1em}

\noindent \textbf{Proof of Theorem \ref{TD}:}

Let $T_0$ be determined in \eqref{T4}, for any $g(t) \in C\left( [T_0,\infty), \mathbb{R}^+ \right)$, define
\begin{equation}\label{M3}
    B(t) := \left\{ \xi \in \mathbb{R}^d : |\xi| \leq g(t) \right\}.
\end{equation}

Using Lemma \ref{TL} and Plancherel's identity, we know that for any $t>T_0$, it holds
\begin{equation}\label{M1}
     \begin{aligned}
         \frac{d}{dt}||w(t)||_{L_x^2}^2 & \leq - \int_{\mathbb{R}^d} |\xi|^{2\al} |\hat{w}(t,\xi)|^2 d\xi + Ct^{-\frac{d+2}{2\al}+1+\frac{2s}{\al}}.
     \end{aligned}
\end{equation}
Split
\begin{equation}\label{M2}
    - \int_{\mathbb{R}^d} |\xi|^{2\al} |\hat{w}(t,\xi)|^2 d\xi = - \int_{B(t)^c} |\xi|^{2\al} |\hat{w}(t,\xi)|^2 d\xi - \int_{B(t)} |\xi|^{2\al} |\hat{w}(t,\xi)|^2 d\xi,
\end{equation}
where $B(t)$ is defined in \eqref{M3}.

Note that
\begin{equation}\label{M4}
    \begin{aligned}
        & - \int_{B(t)^c} |\xi|^{2\al} |\hat{w}(t,\xi)|^2 d\xi 
        \leq - g^{2\al}(t) \int_{B(t)^c} |\hat{w}(t,\xi)|^2 d\xi \\
        & \hspace{.2in} = - g^{2\al}(t) \int_{\mathbb{R}^d} |\hat{w}(t,\xi)|^2 d\xi + g^{2\al}(t) \int_{B(t)} |\hat{w}(t,\xi)|^2 d\xi.
    \end{aligned}
\end{equation}
Plugging \eqref{M2}-\eqref{M4} into \eqref{M1}, we obtain that for any $t>T_0$,
\begin{equation}\label{M5}
    \frac{d}{dt}||w(t)||_{L_x^2}^2 + g^{2\al}(t) \int_{\mathbb{R}^d} |\hat{w}(t,\xi)|^2 d\xi \leq g^{2\al}(t) \int_{B(t)} |\hat{w}(t,\xi)|^2 d\xi + Ct^{-\frac{d+2}{2\al}+1+\frac{2s}{\al}}.
\end{equation}

For any $\tau_2>\tau_1\geq T_0$, integrating on both sides of \eqref{M5} from $\tau_1$ to $\tau_2$, it follows
\begin{align*}
    & ||w(\tau_2)||_{L_x^2}^2 + \int_{\tau_1}^{\tau_2} g^{2\al}(\tau) ||w(\tau)||_{L_x^2}^2 d\tau \\
        \leq & ||w(\tau_1)||_{L_x^2}^2 + \int_{\tau_1}^{\tau_2} \left( g^{2\al}(\tau) \int_{B(\tau)} |\hat{w}(\tau,\xi)|^2 d\xi + C\tau^{-\frac{d+2}{2\al}+1+\frac{2s}{\al}} \right)  d\tau,
\end{align*}
from which and Lemma \ref{L}, one gets that for any $t>T_0$,
\begin{equation}\label{M-2}
    \begin{aligned}
        & ||w(t)||_{L_x^2}^2 e^{\int_{T_0}^t g^{2\al}(\tau) d\tau} \\
        \leq & ||w(T_0)||_{L_x^2}^2 + C \int_{T_0}^t e^{\int_{T_0}^{\tau} g^{2\al}(\tau') d\tau'} \left( g^{2\al}(\tau) \int_{B(\tau)} |\hat{w}(\tau,\xi)|^2 d\xi + \tau^{-\frac{d+2}{2\al}+1+\frac{2s}{\al}} \right) d\tau.
    \end{aligned}
\end{equation}

Using Lemma \ref{PL} and the polar coordinate, we have
\begin{equation}\label{M6}
    \begin{aligned}
        \int_{B(t)} |\hat{w}(t,\xi)|^2 d\xi & \lesssim \left( \left( \int_{T_0}^t ||w(\tau)||_{L_x^2}^2 d\tau \right)^2 + t^{2+\frac{2s}{\al}} \right) \int_{B(t)} |\xi|^2 d\xi \\
        & \lesssim \left( \left( \int_{T_0}^t ||w(\tau)||_{L_x^2}^2 d\tau \right)^2 + t^{2+\frac{2s}{\al}} \right) \int_0^{g(t)} r^{d-1} \cdot r^2 dr \\
        & \lesssim g^{d+2}(t) \left( \left( \int_{T_0}^t ||w(\tau)||_{L_x^2}^2 d\tau \right)^2 + t^{2+\frac{2s}{\al}} \right).
    \end{aligned}
\end{equation}
Putting \eqref{M6} into the right hand side of \eqref{M-2}, one concludes
\begin{equation}\label{M-11}
    \begin{aligned}
        ||w(t)||_{L_x^2}^2 e^{\int_{T_0}^t g^{2\al}(\tau) d\tau}
        \leq ||w(T_0)||_{L_x^2}^2
        + C \int_{T_0}^t e^{\int_{T_0}^{\tau} g^{2\al}(\tau') d\tau'} G(\tau) d\tau,
    \end{aligned}
\end{equation}
where
\begin{equation}\label{M-12}
    G(\tau) = g^{d+2+2\al}(\tau) \left( \left( \int_{T_0}^\tau ||w(\tau')||_{L_x^2}^2 d\tau' \right)^2 + \tau^{2+\frac{2s}{\al}} \right) + \tau^{-\frac{d+2}{2\al}+1+\frac{2s}{\al}}.
\end{equation}

Using the property \eqref{RH1}, we know that
$$||w(t)||_{L_x^2} \leq C, \quad \forall t>T_0,$$
which together with $s<0$ implies
\begin{equation}\label{M-3}
    G(t) \leq g^{d+2+2\al}(t) \cdot t^2 + t^{-\frac{d+2}{2\al}+1+\frac{2s}{\al}}, \quad \forall t>T_0.
\end{equation}

Plugging \eqref{M-3} into \eqref{M-11}, one obtains
\begin{equation}\label{M-4}
    \begin{aligned}
        ||w(t)||_{L_x^2}^2 e^{\int_{T_0}^t g^{2\al}(\tau) d\tau}
        \leq ||w(T_0)||_{L_x^2}^2 + C \int_{T_0}^t e^{\int_{T_0}^{\tau} g^{2\al}(\tau') d\tau'} \left( g^{d+2\al+2}(\tau) \cdot \tau^2 + \tau^{-\frac{d+2}{2\al}+1+\frac{2s}{\al}} \right) d\tau.
    \end{aligned}
\end{equation}

Choosing $g(t)=(\frac{d}{t})^{\frac{1}{2\al}}$, it is obvious that $e^{\int_{T_0}^t g^{2\al}(\tau) d\tau} = (\frac{t}{T_0})^d$, thus the inequality \eqref{M-4} yields
\begin{equation}\label{M-5}
    \begin{aligned}
        ||w(t)||_{L_x^2}^2 (\frac{t}{T_0})^d & \leq ||w(T_0)||_{L_x^2}^2 + C \int_{T_0}^t (\frac{\tau}{T_0})^d \left\{ (\frac{d}{\tau})^{\frac{d+2+2\al}{2\al}} \tau^2 + \tau^{-\frac{d+2}{2\al}+1+\frac{2s}{\al}}  \right\} d\tau \\
        & \lesssim 1 + \int_{T_0}^t \left( \tau^{d-\frac{d+2}{2\al}+1} + \tau^{d-\frac{d+2}{2\al}+1+\frac{2s}{\al}} \right) d\tau \\
        & \lesssim 1 + t^{d-\frac{d+2}{2\al}+2} + t^{d-\frac{d+2}{2\al}+2+\frac{2s}{\al}} \\
        & \lesssim 1 + t^{d-\frac{d+2}{2\al}+2},
    \end{aligned}
\end{equation}
the last inequality holds since $s<0$.

Multiplying $(\frac{t}{T_0})^{-d}$ on both sides of \eqref{M-5} and noting that $-d \leq -\frac{d+2}{2\al}+2$, we conclude
\begin{equation}\label{M-6}
    ||w(t)||_{L_x^2}^2 \lesssim t^{-d} + t^{-\frac{d+2}{2\al}+2} \lesssim t^{-\frac{d+2}{2\al}+2}, \quad \forall t>T_0.
\end{equation}

In order to improve the decay rate \eqref{M-6}, we need the following result, the proof of which shall be given in Appendix \ref{A}.

\begin{lemma}\label{IML}
    Let $T_0$ be determined in \eqref{T4} and $\ga \in (-1,0)$, if one has
    \begin{equation}\label{IM1}
        ||w(t)||_{L_x^2}^2 \lesssim t^\ga, \quad \forall t>T_0.
    \end{equation}
    Then we can improve the decay rate \eqref{IM1} to
    \begin{equation}\label{IM-3}
        ||w(t)||_{L_x^2}^2 \lesssim t^{-\frac{d+2}{2\al}+2+2\ga}+t^{-\frac{d+2}{2\al}+2+\frac{2s}{\al}}, \quad \forall t>T_0.
    \end{equation}
\end{lemma}

Now, let us improve the decay rate \eqref{M-6}, we consider $\al \in \left( \frac{1}{2},\frac{d+2}{4} \right)$ and $\al =\frac{d+2}{4}$ separately.

\vspace{0.5em}

\hspace{.02in}
\noindent \textbf{\underline{$\bullet$ Case 1: $\al \in \left( \frac{1}{2},\frac{d+2}{4} \right)$.}
}

\vspace{0.5em}

In this case we have the following consequence from \eqref{M-6} and Lemma \ref{IML}, which shall be proved in Appendix \ref{A}.

\begin{lemma}\label{2DL}
    Assume that $\al \in \left( \frac{1}{2},\frac{d+2}{4} \right)$, let
    $$A_1^{(1)} := \left\{ (\al,s) \, \Big|\, \al \in \left( \frac{1}{2},\frac{d+2}{6} \right], s \in \Big( -\al+(1-\al)_+,0 \Big) \right\},$$
    and for any integer $n \geq 2$, let
    $$A_n^{(1)} := \left\{ (\al,s) \, \Big|\, \al \in \left( \frac{\si_n(d+2)}{4},\frac{\si_{n+1}(d+2)}{4} \right], s \in \left( -\al+(1-\al)_+,\eta_{n-1}(\al-\frac{d+2}{4}) \right) \right\},$$
    where $\si_n=\frac{2^n-2}{2^n-1}$ and $\eta_n=2^{n+1}-2$. Moreover, for any integer $n \geq 1$, define
    $$A_n^{(2)} := \left\{ (\al,s) \, \Big|\, \al \in \left( \frac{\si_{n+1}(d+2)}{4},\frac{d+2}{4} \right), s \in \left[ \eta_n(\al-\frac{d+2}{4}),\eta_{n-1}(\al-\frac{d+2}{4}) \right) \right\},$$
    and
    $$A_n^{(3)} := \left\{ (\al,s) \, \Big|\, \al \in \left( \frac{\si_{n+1}(d+2)}{4},\frac{d+2}{4} \right), s \in \left( -\al+(1-\al)_+,\eta_n(\al-\frac{d+2}{4}) \right) \right\}.$$
    Then for any $t>T_0$,  the estimate \eqref{M-6} can be improved as
    \begin{equation}\label{2D}
        ||w(t)||_{L_x^2}^2 \lesssim \left\{
        \begin{aligned}
        & t^{-\frac{d+2}{2\al}+2+\frac{2s}{\al}}, \quad \quad \quad \,\,\, (\al,s) \in A_n^{(1)} \cup A_n^{(2)}, \\
        & t^{(2^{n+1}-1)\left( -\frac{d+2}{2\al}+2 \right)}, \quad (\al,s) \in A_n^{(3)}.
        \end{aligned}
        \right.
    \end{equation}
\end{lemma}

Note that
$$\left\{ (\al,s) \, \Big| \, \al \in \left( \frac{1}{2},\frac{d+2}{4} \right), \, s \in (-\al+(1-\al)_+,0) \right\} = \bigcup_{n \geq 1} \left( A_n^{(1)} \cup A_n^{(2)} \right).$$
Thus, we immediately obtain the decay
\begin{equation}\label{IM-4}
    ||w(t)||_{L_x^2}^2 \lesssim t^{-\frac{d+2}{2\al}+2+\frac{2s}{\al}}, \quad \forall t>T_0,
\end{equation}
by using \eqref{2D} directly.

\vspace{0.5em}

\hspace{.02in}
\noindent \textbf{\underline{ $\bullet$ Case 2: $\al=\frac{d+2}{4}$.}
}

\vspace{0.5em}

Note that one can not gain any decay from \eqref{M-6} when $\al=\frac{d+2}{4}$. Thus, we shall proceed in another way to obtain the decay \eqref{IM-4} in this case.

Taking $g(t)=(\frac{2}{3}t\ln{t})^{-\frac{1}{2\al}}$ in \eqref{M-4}, and noting that $e^{\int_{T_0}^t g^{2\al}(\tau) d\tau}=\left( \frac{\ln{t}}{\ln{T_0}} \right)^{\frac{3}{2}}$, we get
\begin{equation}\label{2D1}
    \begin{aligned}
        ||w(t)||_{L_x^2}^2 \left( \frac{\ln{t}}{\ln{T_0}} \right)^{\frac{3}{2}}
        \leq ||w(T_0)||_{L_x^2}^2 + C \int_{T_0}^t \left( \frac{\ln{\tau}}{\ln{T_0}} \right)^{\frac{3}{2}} \left( \tau^{-1} (\ln{\tau})^{-3} + \tau^{-1+\frac{2s}{\al}} \right) d\tau.
    \end{aligned}
\end{equation}
It is obvious that
$$\int_{T_0}^{+\infty} (\ln{\tau})^{\frac{3}{2}} \left( \tau^{-1} (\ln{\tau})^{-3} + \tau^{-1+\frac{2s}{\al}} \right) d\tau < +\infty,$$
which together with \eqref{2D1} yields
\begin{equation}\label{2D2}
    ||w(t)||_{L_x^2}^2 \lesssim (\ln{t})^{-\frac{3}{2}}, \quad \forall t>T_0.
\end{equation}

Using Hölder's inequality and the estimate \eqref{2D2}, it follows
\begin{equation}\label{2D3}
    \int_{T_0}^t ||w(\tau)||_{L_x^2}^2 d\tau \leq (t-T_0)^{\frac{1}{2}} \left( \int_{T_0}^t ||w(\tau)||_{L_x^2}^4 d\tau \right)^{\frac{1}{2}} \lesssim t^{\frac{1}{2}} \left( \int_{T_0}^t ||w(\tau)||_{L_x^2}^2 (\ln{\tau})^{-\frac{3}{2}} d\tau \right)^{\frac{1}{2}}.
\end{equation}
Plugging \eqref{2D3} into \eqref{M-11}-\eqref{M-12}, we have
\begin{equation}\label{2D7}
    \begin{aligned}
        & ||w(t)||_{L_x^2}^2 e^{\int_{T_0}^t g^{2\al}(\tau) d\tau}
        \leq ||w(T_0)||_{L_x^2}^2 \\
        & + C \int_{T_0}^t e^{\int_{T_0}^{\tau} g^{2\al}(\tau') d\tau'} \left\{ g^{6\al}(\tau) \left( \tau \int_{T_0}^\tau ||w(\tau')||_{L_x^2}^2 (\ln{\tau'})^{-\frac{3}{2}} d\tau' + \tau^{2+\frac{2s}{\al}} \right) + \tau^{-1+\frac{2s}{\al}} \right\} d\tau.
    \end{aligned}
\end{equation}

Alternatively, by choosing $g(t)=\left( 2t^{-1} \right)^{\frac{1}{2\al}}$ in \eqref{2D7}, one gets
\begin{align*}
    ||w(t)||_{L_x^2}^2 \left( \frac{t}{T_0} \right)^2
        & \lesssim
        1 + \int_{T_0}^t \tau^2 \left( \tau^{-2} \int_{T_0}^\tau ||w(\tau')||_{L_x^2}^2 (\ln{\tau'})^{-\frac{3}{2}} d\tau' + \tau^{-1+\frac{2s}{\al}}\right) d\tau \\
        & \lesssim 1 + \int_{T_0}^t \left(  \int_{T_0}^\tau ||w(\tau')||_{L_x^2}^2 (\ln{\tau'})^{-\frac{3}{2}} d\tau' + \tau^{1+\frac{2s}{\al}} \right) d\tau,
\end{align*}
from which and $s > -\al$, we obtain
\begin{equation}\label{2D9}
    \begin{aligned}
        ||w(t)||_{L_x^2}^2 \left( \frac{t}{T_0} \right)^2 & \lesssim t^{2+\frac{2s}{\al}} + \int_{T_0}^t \left( \int_{T_0}^\tau ||w(\tau')||_{L_x^2}^2 (\ln{\tau'})^{-\frac{3}{2}} d\tau' \right) d\tau \\
        & \lesssim t^{2+\frac{2s}{\al}} + t \int_{T_0}^t ||w(\tau)||_{L_x^2}^2 (\ln{\tau})^{-\frac{3}{2}} d\tau.
    \end{aligned}
\end{equation}

Multiplying $t^{-2-\frac{s}{\al}}$ on both sides of \eqref{2D9}, and using $s>-\al$, it holds
\begin{align*}
    ||w(t)||_{L_x^2}^2 t^{-\frac{s}{\al}}
        & \lesssim t^{\frac{s}{\al}} + t^{-1-\frac{s}{\al}} \int_{T_0}^t ||w(\tau)||_{L_x^2}^2 (\ln{\tau})^{-\frac{3}{2}} d\tau \\
        & \lesssim t^{\frac{s}{\al}} + \int_{T_0}^t ||w(\tau)||_{L_x^2}^2 \tau^{-\frac{s}{\al}} \tau^{-1} (\ln{\tau})^{-\frac{3}{2}} d\tau.
\end{align*}
Thus, by using Grönwall's inequality, one concludes
\begin{equation}\label{2D10}
    \begin{aligned}
        ||w(t)||_{L_x^2}^2 t^{-\frac{s}{\al}} 
        & \lesssim t^{\frac{s}{\al}} + e^{C_2} \int_{T_0}^t \tau^{-1+\frac{s}{\al}} (\ln{\tau})^{-\frac{3}{2}} d\tau \leq C,
    \end{aligned}
\end{equation}
by using $s<0$, where
$$C_2 = \int_{T_0}^{+\infty} \tau^{-1} (\ln{\tau})^{-\frac{3}{2}} d\tau < +\infty.$$

The estimate \eqref{2D10} implies
$$||w(t)||_{L_x^2}^2 \lesssim t^{\frac{s}{\al}}, \quad \forall t>T_0,$$
from which we obtain the decay \eqref{IM-4} for $\al=\frac{d+2}{4}$, by using Lemma \ref{IML} with $\ga=\frac{s}{\al}$.

Combining the decays \eqref{IM-4} for $||w||_{L_x^2}$ and \eqref{DH} for $||h^\om||_{L_x^2}$, we immediately conclude the time decay \eqref{TD1} for $||u||_{L_x^2}$. The proof is completed. 

\qed

\begin{remark}
    The calculations in this section are formal. However, they make sense for the approximation solutions $\{w^{N_k}\}_{k \in \mathbb{N}}$, which are constructed in Theorem \ref{DT}. More precisely, there exists a positive constant $C$, which is independent of $k$, such that
 \begin{equation}\label{RE}
        ||w^{N_k}||_{L^2(\mathbb{R}^d)}^2 \leq C t^{-\frac{d+2}{2\al}+2+\frac{2s}{\al}}, \quad \forall t>T_0.
    \end{equation}
    Let $\{D_j\}_{j \in \mathbb{N}}$ be a sequence of disjoint bounded domains in $\mathbb{R}^d$, satisfying $\bigcup_{j \in \mathbb{N}} D_j = \mathbb{R}^d$, from \eqref{RE} we obtain a subsequence, still denoted by $\{w^{N_k}\}_{k \in \mathbb{N}}$, and a function $\widetilde{\widetilde{w}}$, such that for any $t>T_0$, as $k\to +\infty$, 
    $$w^{N_k}(t) \rightharpoonup \widetilde{\widetilde{w}}(t), \quad \mathrm{in} ~\  L^2(D_j), ~\ \forall j \in \mathbb{N},$$
    from which and the convergence \eqref{L1}, one gets that there exists a zero measure set $N$, such that for any $t \in \{ \tau>T_0 : \tau \notin N \}$, as $k\to +\infty$, 
    $$w^{N_k}(t) \rightharpoonup w_2(t), \quad \mathrm{in} ~\  L^2(D_j), ~\ \forall j \in \mathbb{N},$$
    where $w_2$ is the weak solution constructed in Theorem \ref{DT}. Thus, by using the lower semi-continuity for weak convergence, Fatou's lemma and the estimate \eqref{RE}, it follows that for any $t \in \{ \tau>T_0 : \tau \notin N \}$,
    \begin{equation}\label{RE1}
        \begin{aligned}
            & ||w_2(t)||_{L^2(\mathbb{R}^d)}^2 = \sum_{j \in \mathbb{N}} ||w_2(t)||_{L^2(D_j)}^2 \leq \sum_{j \in \mathbb{N}} \liminf_{k \rightarrow +\infty} ||w^{N_k}(t)||_{L^2(D_j)}^2 \\
            & \hspace{.2in }\leq \liminf_{k \rightarrow +\infty} \sum_{j \in \mathbb{N}}  ||w^{N_k}(t)||_{L^2(D_j)}^2 = \liminf_{k \rightarrow +\infty} ||w^{N_k}(t)||_{L^2(\mathbb{R}^d)}^2 \leq C t^{-\frac{d+2}{2\al}+2+\frac{2s}{\al}}.
        \end{aligned}
    \end{equation}
    Next, for any fixed $t \in \{ \tau>T_0 : \tau \in N \}$, choosing $\{\tau_n\}_{n \in \mathbb{N}} \subset \{ \tau>T_0 : \tau \notin N \}$ such that $\tau_n \xrightarrow{n \rightarrow +\infty} t$, by using the weak continuity \eqref{L3} of $w_2$, we conclude
    \begin{align*}
        & ||w_2(t)||_{L^2(\mathbb{R}^d)}^2 \leq \liminf_{\tau \rightarrow t} ||w_2(\tau)||_{L^2(\mathbb{R}^d)}^2 \leq \liminf_{n \rightarrow +\infty} ||w_2(\tau_n)||_{L^2(\mathbb{R}^d)}^2 \\
            & \hspace{.2in} \leq C \liminf_{n \rightarrow +\infty} {\tau_n}^{-\frac{d+2}{2\al}+2+\frac{2s}{\al}} = C t^{-\frac{d+2}{2\al}+2+\frac{2s}{\al}},
    \end{align*}
    from which and \eqref{RE1}, we obtain that the decay \eqref{RE} holds true for the weak solution $w(t)$ defined in \eqref{WS20}.
\end{remark}

\appendix

\section{Proof of several Lemmas}\label{A}

\numberwithin{equation}{section}

\setcounter{equation}{0}

In this section, we give the proofs of several lemmas used in the previous sections.

\vspace{0.5em}

\noindent \textbf{Proof of Lemma 3.2:}

 On one hand, for the term $I_1$ defined in \eqref{ME11}, by using the inequality \eqref{IH}, one obtains
 \begin{align*}
     ||I_1||_{L_{\ka;T}^m W_x^{\be,r}} & = \left|\left| \int_0^\frac{t}{2} t^\ka e^{-(t-\tau)(-\Delta)^\al} (1-\Delta)^{\frac{\be+\al(2\mu+1)}{2}} (1-\Delta)^{-\frac{\al(2\mu+1)}{2}} B(f,g) \, d\tau  \right|\right|_{L_T^m L_x^r} \\
     & \lesssim \left|\left| \int_0^\frac{t}{2} t^\ka \left( 1+(t-\tau)^{\frac{\be+\al(2\mu+1)}{2\al}} \right) (t-\tau)^{-\frac{\be+\al(2\mu+1)}{2\al} - \frac{d}{2\al}\left( \frac{1}{p} \right)} ||B(f,g)||_{W_x^{-\al(2\mu+1),\frac{pr}{p+r}}} \, d\tau  \right|\right|_{L_T^m} \\
     & \lesssim \left( 1+T^{\frac{\be+\al(2\mu+1)}{2\al}} \right) \left|\left| \int_0^\frac{t}{2} t^\ka  (t-\tau)^{-\frac{\be+\al(2\mu+1)}{2\al} - \frac{d}{2\al}\left( \frac{1}{p} \right)} ||B(f,g)||_{W_x^{-\al(2\mu+1),\frac{pr}{p+r}}} \, d\tau  \right|\right|_{L_T^m} \\
     & \lesssim \left|\left| \int_0^\frac{t}{2} (t-\tau)^{\ka -\frac{\be+\al(2\mu+1)}{2\al} - \frac{d}{2\al}\left( \frac{1}{p} \right)} ||B(f,g)||_{W_x^{-\al(2\mu+1),\frac{pr}{p+r}}} \, d\tau  \right|\right|_{L_T^m},
 \end{align*}
 the last inequality holds since $T\in (0, 1)$ and $\tau \in \left[ 0,\frac{t}{2} \right]$. Noting from the relation \eqref{AP} and the hypothesis \eqref{H3.2} that
 $$-\ka+\frac{\be+\al(2\mu+1)}{2\al}+\frac{d}{2\al}\cdot\frac{1}{p} = 1-\left( \frac{1}{2}-\frac{1}{m} \right),$$
 by using Lemma \ref{EY}, the estimate \eqref{M11} and the relation $\frac{1}{a}+\frac{1}{b}=\frac{1}{2}$, we have
 \begin{equation}\label{ME12}
     \begin{aligned}
         ||I_1||_{L_{\ka;T}^m W_x^{\be,r}} & \lesssim \Big|\Big| ||B(f,g)||_{W_x^{-\al(2\mu+1),\frac{pr}{p+r}}} \Big|\Big|_{L_T^2} \\
         & \lesssim \Big|\Big| ||f||_{W_x^{1-\al(2\mu+1),p}} ||g||_{L_x^r} + ||f||_{L_x^r} ||g||_{W_x^{1-\al(2\mu+1),p}} \Big|\Big|_{L_T^2} \\
         & \lesssim ||f||_{L_T^b W_x^{1-\al(2\mu+1),p}} ||g||_{L_T^a L_x^r} + ||f||_{L_T^a L_x^r} ||g||_{L_T^b W_x^{1-\al(2\mu+1),p}}.
     \end{aligned}
 \end{equation}

 On the other hand, to estimate the term $I_2$ given in \eqref{ME11}, using the inequality \eqref{IH} and $T<1$, it follows
 \begin{align*}
     ||I_2||_{L_{\ka;T}^m W_x^{\be,r}} & = \left|\left| \int_\frac{t}{2}^t t^\ka e^{-(t-\tau)(-\Delta)^\al} (1-\Delta)^{\frac{2\be+1}{4}} (1-\Delta)^{-\frac{1}{4}} B(f,g) \, d\tau \right|\right|_{L_T^m L_x^r} \\
     & \lesssim \left|\left| \int_\frac{t}{2}^t t^\ka \left( 1+(t-\tau)^{\frac{\be+\frac{1}{2}}{2\al}} \right) (t-\tau)^{-\frac{\be+\frac{1}{2}}{2\al}-\frac{1}{2\al}\left( \frac{1}{p} \right)} ||B(f,g)||_{W_x^{-\frac{1}{2},\frac{pr}{p+r}}} \, d\tau \right|\right|_{L_T^m} \\
     & \lesssim \left( 1+T^{\frac{\be+\frac{1}{2}}{2\al}} \right) \left|\left| \int_\frac{t}{2}^t t^\ka (t-\tau)^{-\frac{\be+\frac{1}{2}}{2\al}-\frac{1}{2\al}\left( \frac{1}{p} \right)} ||B(f,g)||_{W_x^{-\frac{1}{2},\frac{pr}{p+r}}} \, d\tau \right|\right|_{L_T^m} \\
     & \lesssim \left|\left| \int_\frac{t}{2}^t t^{\ka-\left(\frac{3\mu+1}{4}-\frac{1}{8\al}\right)} (t-\tau)^{-\frac{\be+\frac{1}{2}}{2\al}-\frac{1}{2\al}\left( \frac{1}{p} \right)} \tau^{\frac{3\mu+1}{4}-\frac{1}{8\al}} ||B(f,g)||_{W_x^{-\frac{1}{2},\frac{pr}{p+r}}} \, d\tau \right|\right|_{L_T^m} \\
     & \lesssim \left|\left| \int_\frac{t}{2}^t (t-\tau)^{\ka-\left(\frac{3\mu+1}{4}-\frac{1}{8\al}\right)-\frac{\be+\frac{1}{2}}{2\al}-\frac{1}{2\al}\left( \frac{1}{p} \right)} \tau^{\frac{3\mu+1}{4}-\frac{1}{8\al}} ||B(f,g)||_{W_x^{-\frac{1}{2},\frac{pr}{p+r}}} \, d\tau \right|\right|_{L_T^m},
 \end{align*}
by using $\tau \in \Big[ \frac{t}{2},t \Big]$ and $\ka \in \left[ 0,\frac{3\mu+1}{4}-\frac{1}{8\al} \right]$. Moreover, by noting
$$-\ka+\left(\frac{3\mu+1}{4}-\frac{1}{8\al}\right)+\frac{\be+\frac{1}{2}}{2\al}+\frac{1}{2\al}\cdot \frac{1}{p}  = 1-\left( \frac{1-2\al(\mu-1)}{8\al} + \frac{1}{a} - \frac{1}{m} \right),$$
using Lemma \ref{EY} and the estimate \eqref{M11}, one gets
$$||I_2||_{L_{\ka;T}^m W_x^{\be,r}} \lesssim \Big|\Big| \tau^{\frac{3\mu+1}{4}-\frac{1}{8\al}} \left( ||f||_{W_x^{\frac{1}{2},p}} ||g||_{L_x^r} + ||g||_{L_x^r} ||f||_{W_x^{\frac{1}{2},p}} \right) \Big|\Big|_{L_T^l},$$
where $\frac{1}{l} = \frac{1-2\al(\mu-1)}{8\al} + \frac{1}{a}$, so by using Hölder's inequality, we obtain
\begin{equation}\label{ME13}
    \begin{aligned}
        ||I_2||_{L_{\ka;T}^m W_x^{\be,r}} & \lesssim ||f||_{L_{\frac{3\mu+1}{4}-\frac{1}{8\al};T}^{\frac{8\al}{1-2\al(\mu-1)}} W_x^{\frac{1}{2},p}} ||g||_{L_T^a L_x^r} + ||f||_{L_T^a L_x^r} ||g||_{L_{\frac{3\mu+1}{4}-\frac{1}{8\al};T}^{\frac{8\al}{1-2\al(\mu-1)}} W_x^{\frac{1}{2},p}}.
    \end{aligned}
\end{equation}

Combining \eqref{ME12}-\eqref{ME13} we conclude the estimate \eqref{ME}.

\qed

\vspace{0.5em}

\noindent \textbf{Proof of Lemma 3.3:}

    The proof is similar to that of Lemma \ref{MEL}, we present the detail here for complement. 
    
    For $I_1$ given in the decomposition \eqref{ME11}, using the estimate \eqref{IH} and $T<1$, it follows
    \begin{align*}
     ||I_1||_{L_{\ka;T}^m W_x^{\be,r}} & = \left|\left| \int_0^\frac{t}{2} t^\ka e^{-(t-\tau)(-\Delta)^\al} (1-\Delta)^{\frac{1+\be}{2}} (1-\Delta)^{-\frac{1}{2}} B(f,g) \, d\tau  \right|\right|_{L_T^m L_x^r} \\
     & \lesssim \left|\left| \int_0^\frac{t}{2} t^\ka \left( 1+(t-\tau)^{\frac{1+\be}{2\al}} \right) (t-\tau)^{-\frac{1+\be}{2\al} - \frac{d}{2\al}\left( \frac{1}{p} \right)} ||B(f,g)||_{W_x^{-1,\frac{pr}{p+r}}} \, d\tau  \right|\right|_{L_T^m} \\
     & \lesssim \left( 1+T^{\frac{1+\be}{2\al}} \right) \left|\left| \int_0^\frac{t}{2} t^\ka  (t-\tau)^{-\frac{1+\be}{2\al} - \frac{d}{2\al}\left( \frac{1}{p} \right)} ||B(f,g)||_{W_x^{-1,\frac{pr}{p+r}}} \, d\tau  \right|\right|_{L_T^m} \\
     & \lesssim \left|\left| \int_0^\frac{t}{2} (t-\tau)^{\ka -\frac{1+\be}{2\al} - \frac{d}{2\al}\left( \frac{1}{p} \right)} ||B(f,g)||_{W_x^{-1,\frac{pr}{p+r}}} \, d\tau  \right|\right|_{L_T^m}.
 \end{align*}
 Noting from the relation \eqref{AP} and the hypothesis \eqref{H3.3} that
 $$-\ka + \frac{1+\be}{2\al} + \frac{d}{2\al}\cdot \frac{1}{p}  = 1- \left( \frac{2}{a} - \frac{1}{m} \right),$$
 by using Lemma \ref{EY}, the estimate \eqref{M11} and Hölder's inequality, one has
 \begin{equation}\label{ME22}
     \begin{aligned}
         ||I_1||_{L_{\ka;T}^m W_x^{\be,r}} & \lesssim \Big|\Big| ||B(f,g)||_{W_x^{-1,\frac{pr}{p+r}}} \Big|\Big|_{L_T^{\frac{a}{2}}} \\
         & \lesssim \Big|\Big| ||f||_{L_x^p} ||g||_{L_x^r} + ||f||_{L_x^r} ||g||_{L_x^p} \Big|\Big|_{L_T^{\frac{a}{2}}} \\
         & \lesssim ||f||_{L_T^a L_x^p} ||g||_{L_T^a L_x^r} + ||f||_{L_T^a L_x^r} ||g||_{L_T^a L_x^p}.
     \end{aligned}
 \end{equation}

Similarly, for the term $I_2$ defined in \eqref{ME11}, by using the inequality \eqref{IH} and $T<1$, we obtain
\begin{align*}
     ||I_2||_{L_{\ka;T}^m W_x^{\be,r}} & = \left|\left| \int_\frac{t}{2}^t t^\ka e^{-(t-\tau)(-\Delta)^\al} (1-\Delta)^{\frac{\be+2(1-\al)}{2}} (1-\Delta)^{-\frac{2(1-\al)}{2}} B(f,g) \, d\tau \right|\right|_{L_T^m L_x^r} \\
     & \lesssim \left|\left| \int_\frac{t}{2}^t t^\ka \left( 1+(t-\tau)^{\frac{\be+2(1-\al)}{2\al}} \right) (t-\tau)^{-\frac{\be+2(1-\al)}{2\al}-\frac{1}{2\al}\left( \frac{1}{p} \right)} ||B(f,g)||_{W_x^{-2(1-\al),\frac{pr}{p+r}}} \, d\tau \right|\right|_{L_T^m} \\
     & \lesssim \left( 1+T^{\frac{\be+2(1-\al)}{2\al}} \right) \left|\left| \int_\frac{t}{2}^t t^\ka (t-\tau)^{-\frac{\be+2(1-\al)}{2\al}-\frac{1}{2\al}\left( \frac{1}{p} \right)} ||B(f,g)||_{W_x^{-2(1-\al),\frac{pr}{p+r}}} \, d\tau \right|\right|_{L_T^m} \\
     & \lesssim \left|\left| \int_\frac{t}{2}^t t^{\ka-\left(1-\frac{1}{2\al}\right)} (t-\tau)^{-\frac{\be+2(1-\al)}{2\al}-\frac{1}{2\al}\left( \frac{1}{p} \right)} \tau^{1-\frac{1}{2\al}} ||B(f,g)||_{W_x^{-2(1-\al),\frac{pr}{p+r}}} \, d\tau \right|\right|_{L_T^m} \\
     & \lesssim \left|\left| \int_\frac{t}{2}^t (t-\tau)^{\ka-\left(1-\frac{1}{2\al}\right)-\frac{\be+2(1-\al)}{2\al}-\frac{1}{2\al}\left( \frac{1}{p} \right)} \tau^{1-\frac{1}{2\al}} ||B(f,g)||_{W_x^{-2(1-\al),\frac{pr}{p+r}}} \, d\tau \right|\right|_{L_T^m},
 \end{align*}
 by noting
 $$-\ka+\left(1-\frac{1}{2\al}\right)+\frac{\be+2(1-\al)}{2\al}+\frac{d}{2\al}\cdot \frac{1}{p}  = 1 - \left( \frac{2}{a} - \frac{1}{m} \right),$$
 using Lemma \ref{EY} and the estimate \eqref{M11}, one gets
 \begin{align*}
     ||I_2||_{L_{\ka;T}^m W_x^{\be,r}} 
        & \lesssim \Big|\Big| \tau^{1-\frac{1}{2\al}} \left( ||f||_{W_x^{2\al-1,p}} ||g||_{L_x^r} + ||f||_{L_x^r} ||g||_{W_x^{2\al-1,p}} \right) \Big|\Big|_{L_T^{\frac{a}{2}}} \\
        & \lesssim ||f||_{L_{1-\frac{1}{2\al};T}^a W_x^{2\al-1,p}} ||g||_{L_x^a L_x^r} + ||f||_{L_x^a L_x^r} ||g||_{L_{1-\frac{1}{2\al};T}^a W_x^{2\al-1,p}},
 \end{align*}
from which and \eqref{ME22} we conclude the estimate \eqref{ME2}.

\qed

\vspace{0.5em}

\noindent \textbf{Proof of Lemma \ref{L}:}

Fix $t>t_0$, for any $n \in \mathbb{N}$, let $t_j=t_0+\frac{j}{n}(t-t_0)$ with $j=0,1,\cdots,n$. Then for any $1 \leq j \leq n$, one has
    \begin{equation}\label{L-3}
        \begin{aligned}
            & x(t_j) e^{\int_{t_0}^{t_j} g^{2\al}(\tau) d\tau} - x(t_{j-1}) e^{\int_{t_0}^{t_{j-1}} g^{2\al}(\tau) d\tau} \\
            = & x(t_j) e^{\int_{t_0}^{t_j} g^{2\al}(\tau) d\tau} -x(t_j) e^{\int_{t_0}^{t_{j-1}} g^{2\al}(\tau) d\tau} + x(t_j) e^{\int_{t_0}^{t_{j-1}} g^{2\al}(\tau) d\tau} - x(t_{j-1}) e^{\int_{t_0}^{t_{j-1}} g^{2\al}(\tau) d\tau} \\
            = & x(t_j) e^{\int_{t_0}^{t_{j-1}} g^{2\al}(\tau) d\tau} \left( e^{\int_{t_{j-1}}^{t_j} g^{2\al}(\tau) d\tau} - 1 \right) + e^{\int_{t_0}^{t_{j-1}} g^{2\al}(\tau) d\tau} \left( x(t_j) - x(t_{j-1}) \right).
        \end{aligned}
    \end{equation}
    Using Taylor's formula and the Mean Value Theorem, for $n$ large enough, we get
    \begin{equation}\label{L-2}
        e^{\int_{t_{j-1}}^{t_j} g^{2\al}(\tau) d\tau} - 1 = \int_{t_{j-1}}^{t_j} g^{2\al}(\tau) d\tau + C_{t} \, O(\frac{1}{n^2}),
    \end{equation}
    by using the continuity of $g$. Plugging \eqref{L-2} into \eqref{L-3}, one obtains
    \begin{equation}\label{L-6}
        \begin{aligned}
            & \mathrm{LHS} ~\  \mathrm{of} ~\ \eqref{L-3} \\
            = & x(t_j) e^{\int_{t_0}^{t_{j-1}} g^{2\al}(\tau) d\tau} \int_{t_{j-1}}^{t_j} g^{2\al}(\tau) d\tau + C_{t} \, O(\frac{1}{n^2}) + e^{\int_{t_0}^{t_{j-1}} g^{2\al}(\tau) d\tau} \left( x(t_j) - x(t_{j-1}) \right) \\
            = & e^{\int_{t_0}^{t_{j-1}} g^{2\al}(\tau) d\tau} \int_{t_{j-1}}^{t_j} \left( x(t_j)-x(\tau) \right) g^{2\al}(\tau) d\tau + C_{t} \, O(\frac{1}{n^2}) \\
            & + e^{\int_{t_0}^{t_{j-1}} g^{2\al}(\tau) d\tau} \left( x(t_j) - x(t_{j-1}) + \int_{t_{j-1}}^{t_j} x(\tau) g^{2\al}(\tau) d\tau \right).
        \end{aligned}
    \end{equation}
    Using the hypothesis \eqref{L-1}, we know that
    \begin{equation}\label{L-4}
        x(t_j) - x(t_{j-1}) + \int_{t_{j-1}}^{t_j} x(\tau) g^{2\al}(\tau) d\tau \leq \int_{t_{j-1}}^{t_j} y(\tau) d\tau,
    \end{equation}
    and
    \begin{equation}\label{L-5}
        \begin{aligned}
            \int_{t_{j-1}}^{t_j} \left( x(t_j)-x(\tau) \right) g^{2\al}(\tau) d\tau
            & \leq \int_{t_{j-1}}^{t_j} \left(  \int_{\tau}^{t_j} y(\tau') d\tau' \right) g^{2\al}(\tau) d\tau \\
            & \leq C_t \, \frac{1}{n} \int_{t_{j-1}}^{t_j} y(\tau) d\tau.
        \end{aligned}
    \end{equation}
    Plugging \eqref{L-4}-\eqref{L-5} into the right hand side of \eqref{L-6}, it follows
    \begin{align*}
        & x(t_j) e^{\int_{t_0}^{t_j} g^{2\al}(\tau) d\tau} - x(t_{j-1}) e^{\int_{t_0}^{t_{j-1}} g^{2\al}(\tau) d\tau} \\
            \leq & C_t \, \frac{1}{n} \int_{t_{j-1}}^{t_j} y(\tau) d\tau + C_{t} \, O(\frac{1}{n^2}) +  \, e^{\int_{t_0}^{t_{j-1}} g^{2\al}(\tau) d\tau} \int_{t_{j-1}}^{t_j} y(\tau) d\tau \\
            \leq & C_t \, \frac{1}{n} \int_{t_{j-1}}^{t_j} y(\tau) d\tau + C_{t} \, O(\frac{1}{n^2}) + \,  \int_{t_{j-1}}^{t_j} e^{\int_{t_0}^{\tau} g^{2\al}(\tau) d\tau} y(\tau) d\tau.
    \end{align*}
    Summing up one concludes
    \begin{align*}
        x(t) e^{\int_{t_0}^{t} g^{2\al}(\tau) d\tau} - x(t_0) & = \sum_{j=1}^n \left( x(t_j) e^{\int_{t_0}^{t_j} g^{2\al}(\tau) d\tau} - x(t_{j-1}) e^{\int_{t_0}^{t_{j-1}} g^{2\al}(\tau) d\tau} \right) \\
            & \leq C_t \, \frac{1}{n} \int_{t_0}^{t} y(\tau) d\tau + C_{t} \, O(\frac{1}{n}) +   \int_{t_0}^{t} e^{\int_{t_0}^{\tau} g^{2\al}(\tau) d\tau} y(\tau) d\tau.
    \end{align*}
    Thus, we obtain \eqref{L-7} by passing $n \rightarrow \infty$.
    
\qed

\vspace{0.5em}

\noindent \textbf{Proof of Lemma \ref{IML}:}

Using the hypothesis \eqref{IM1}, from the definition of $G(t)$ given in \eqref{M-12}, we have
\begin{equation}\label{IM-1}
    G(t) \leq g^{d+2\al+2}(t) \left( t^{2+2\ga}+t^{2+\frac{2s}{\al}} \right) + t^{-\frac{d+2}{2\al}+1+\frac{2s}{\al}}, \quad \forall t>T_0.
\end{equation}
Plugging \eqref{IM-1} into the right hand side of \eqref{M-11}, one obtains
\begin{equation}\label{IM-2}
    \begin{aligned}
        & ||w(t)||_{L_x^2}^2 e^{\int_{T_0}^t g^{2\al}(\tau) d\tau}
        \leq ||w(T_0)||_{L_x^2}^2 \\
        & \hspace{.2in} + C \int_{T_0}^t e^{\int_{T_0}^{\tau} g^{2\al}(\tau') d\tau'} \left\{ g^{d+2\al+2}(\tau) \left( \tau^{2+2\ga}+\tau^{2+\frac{2s}{\al}} \right) + \tau^{-\frac{d+2}{2\al}+1+\frac{2s}{\al}} \right\} d\tau.
    \end{aligned}
\end{equation}

Similarly, taking $g(t)=(\frac{2d}{t})^{\frac{1}{2\al}}$ in \eqref{IM-2}, it follows
\begin{equation}
    \begin{aligned}
        & ||w(t)||_{L_x^2}^2 (\frac{t}{T_0})^{2d} \\
        \lesssim & ||w(T_0)||_{L_x^2}^2 + \int_{T_0}^t (\frac{\tau}{T_0})^{2d} \left\{ (\frac{d}{\tau})^{\frac{d+2\al+2}{2\al}} \left( \tau^{2+2\ga} + \tau^{2+\frac{2s}{\al}} \right) + \tau^{-\frac{d+2}{2\al}+1+\frac{2s}{\al}}  \right\} d\tau \\
        \lesssim & 1 + \int_{T_0}^t \left( \tau^{2d-\frac{d+2}{2\al}+1+2\ga} + \tau^{2d-\frac{d+2}{2\al}+1+\frac{2s}{\al}} \right) d\tau \\
        \lesssim & 1 + t^{2d-\frac{d+2}{2\al}+2+2\ga} + t^{2d-\frac{d+2}{2\al}+2+\frac{2s}{\al}},
    \end{aligned}
\end{equation}
from which we immediately conclude \eqref{IM-3} by multiplying $(\frac{t}{T_0})^{-2d}$ on both sides.

\qed

\noindent \textbf{Proof of Lemma \ref{2DL}:}

We first show \eqref{2D} for $n=1$. It is obvious that $-\frac{d+2}{2\al}+2 \leq \frac{s}{\al}$ for any $(\al,s) \in A_1^{(1)} \cup A_1^{(2)}$, thus one gets
\begin{equation}\label{2D-1}
    ||w(t)||_{L_x^2}^2 \lesssim t^\frac{s}{\al}, \quad (\al,s) \in A_1^{(1)} \cup A_1^{(2)},
\end{equation}
by using \eqref{M-6} directly. Moreover, by using \eqref{2D-1} and Lemma \ref{IML} with $\ga=\frac{s}{\al}$, we have
\begin{equation}\label{2D-11}
    ||w(t)||_{L_x^2}^2 \lesssim t^{-\frac{d+2}{2\al}+2+\frac{2s}{\al}}, \quad (\al,s) \in A_1^{(1)} \cup A_1^{(2)}.
\end{equation}

Alternatively, for any $(\al,s) \in A_1^{(3)}$, one has $-\frac{d+2}{2\al}+2 \in (-1,0)$, so by using Lemma \ref{IML} with $\ga=-\frac{d+2}{2\al}+2$, it follows
\begin{equation}\label{2D-2}
    ||w(t)||_{L_x^2} \lesssim t^{3\left( -\frac{d+2}{2\al}+2 \right)} + t^{-\frac{d+2}{2\al}+2 + \frac{2s}{\al}}  \lesssim t^{3\left( -\frac{d+2}{2\al}+2 \right)}, \quad (\al,s) \in A_1^{(3)},
\end{equation}
by using $\frac{s}{\al}<-\frac{d+2}{2\al}+2$ in this set.

Combining \eqref{2D-11}-\eqref{2D-2} one obtains that the decay rate \eqref{2D} holds true for $n=1$. Next, let us show the decay rate \eqref{2D} for any integer $n \geq 1$ by induction.

Assume that \eqref{2D} holds for $n=k$. For $n=k+1$, note that
$$A_k^{(3)} = A_{k+1}^{(1)} \cup A_{k+1}^{(2)} \cup A_{k+1}^{(3)},$$
thus, from the assumption we know
\begin{equation}\label{2D-3}
    ||w(t)||_{L_x^2}^2 \lesssim t^{(2^{k+1}-1)\left( -\frac{d+2}{2\al}+2 \right)}, \quad (\al,s) \in \bigcup_{j=1}^3 A_{k+1}^{(j)}.
\end{equation}
Clearly, one has
$$(2^{k+1}-1)\left( -\frac{d+2}{2\al}+2 \right) \leq \frac{s}{\al}, \quad (\al,s) \in A_{k+1}^{(1)} \cup A_{k+1}^{(2)},$$
which together with \eqref{2D-3} implies
$$||w(t)||_{L_x^2}^2 \lesssim t^{\frac{s}{\al}}, \quad (\al,s) \in A_{k+1}^{(1)} \cup A_{k+1}^{(2)},$$
from which and Lemma \ref{IML} we obtain
\begin{equation}\label{2D-4}
    ||w(t)||_{L_x^2}^2 \lesssim t^{-\frac{d+2}{2\al}+2+\frac{2s}{\al}}, \quad (\al,s) \in A_{k+1}^{(1)} \cup A_{k+1}^{(2)}.
\end{equation}

On the other hand, for any $(\al,s) \in A_{k+1}^{(3)}$, one has $\frac{s}{\al} < (2^{k+1}-1)\left( -\frac{d+2}{2\al}+2 \right)$, hence we get
\begin{equation}\label{2D-22}
    ||w(t)||_{L_x^2}^2 \lesssim t^{(2^{k+2}-1)\left( -\frac{d+2}{2\al}+2 \right)} + t^{-\frac{d+2}{2\al}+2 + \frac{2s}{\al}} \lesssim t^{(2^{k+2}-1)\left( -\frac{d+2}{2\al}+2 \right)}, \quad (\al,s) \in A_{k+1}^{(3)},
\end{equation}
by using Lemma \ref{IML} with $\ga=(2^{k+1}-1)\left( -\frac{d+2}{2\al}+2 \right)$.

Combining \eqref{2D-4}-\eqref{2D-22} we conclude that the decay rate \eqref{2D} holds true for $n=k+1$, thus it holds true for any integer $n \geq 1$ by induction.

\qed

\vspace{.1in}
\par{\bf Acknowledgements.} This research was supported by National Key R\&D Program of China under grant 2024YFA1013302, and National Natural Science Foundation of China under grant Nos.12331008, 12171317 and 12250710674.

\end{document}